\documentclass[12pt]{article}
\usepackage[dvipsnames]{xcolor}
\usepackage{natbib}
\usepackage{graphicx} 
\usepackage{fullpage,float}
\usepackage{amsmath,amssymb,amsthm}
\usepackage[normalem]{ulem}
\usepackage[colorlinks]{hyperref}
\usepackage{authblk}
\usepackage{amssymb,epsfig,graphics}
\usepackage{amsmath}
\usepackage{dsfont}
\usepackage{enumerate}
\usepackage{tikz}
\usepackage{comment}
\usetikzlibrary{arrows}

\def\L{\mathcal{L}}

\newcommand\cM{{\mathcal M}}
\newcommand\cT{{\mathcal T}}
\newcommand{\eps}{\varepsilon}
\theoremstyle{plain}

\newtheorem{theorem}{Theorem}[section]

\newtheorem{lemma}[theorem]{Lemma}
\newtheorem{remark}[theorem]{Remark}
\newtheorem{proposition}[theorem]{Proposition}

\title{Fixed-point approximation for\\ self-consistent transfer operators with Newton's method}
\author[1]{Wael Bahsoun}
\author[2]{Gary Froyland}
\author[2]{Maxence Phalempin} 

\affil[1]{Department of Mathematical Sciences\authorcr Loughborough University, Leicestershire LE11 3TU, UK}
\affil[2]{School of Mathematics and Statistics\authorcr UNSW Sydney, Sydney NSW 2052, Australia}

\date{}

\begin{document}

\maketitle
\begin{abstract}
  Self-consistent transfer operators arise naturally in the study of mean-field coupled dynamical systems and are closely related to kinetic PDEs such as the Vlasov equation. Despite substantial progress on existence and uniqueness of fixed points for self-consistent transfer operators, the development of fast, reliable, and provably accurate numerical methods remains largely unresolved. In this work, we construct a nonlinear Fourier–Fejér discretisation and establish convergence of the resulting finite-dimensional fixed point to that of the original self-consistent transfer operator. Further, using the nonlinear Fourier–Fejér discretisation, we prove exponential convergence of a sequential iteration scheme and develop a Newton framework with quadratic convergence. We present numerical examples demonstrating the efficiency and flexibility of the above methods.
\end{abstract}
\section{Introduction}
Self-consistent transfer operators are \emph{nonlinear} transfer operators. They are considered as discrete-time analogues of the Vlasov equation \cite{Vla68}, whose general long-term behaviour is far from being understood \cite{CelVil11}. This relation between self-consistent transfer operators and PDEs has generated important new connections between semi-classical analysis \cite{chabuetetal24} and dynamical systems \cite{BL25,bahetal23}. In the field of dynamical systems, self-consistent transfer operators arise from considering the thermodynamic limit of the evolution of states (probability measures) in mean-field coupled dynamical systems. Such systems are of paramount importance for applications in physical  and biological sciences, see \cite{DBB23, SP24, tanzi23} and references therein.
 
There is a clear connection between infinite-dimensional mean-field coupled systems and self-consistent transfer operators: invariant states of the coupled system are fixed points of the self-consistent transfer operator. In particular, \emph{attracting} fixed points of self-consistent transfer operators, when they exist, correspond to attracting invariant states
of the corresponding infinite-dimensional mean field dynamical system \cite{BKZ09}.
Starting with the work of Keller \cite{keller00}, several authors studied existence and uniqueness of attracting fixed points of self-consistent transfer operators \cite{BK24,bahetal23,balintetal18,ST21,galatolo22}. However, as is the case for any evolution-type operator, proving the existence of a fixed point of a self-consistent transfer operator does not directly lead to an explicit solution for the fixed point. 
Thus, it is important to develop techniques to find approximate solutions in an appropriate topology. For classical \emph{linear} transfer operators, this is a very old problem going back to Ulam \cite{Ul60}, who suggested to partition the phase space of the dynamical system into a finite number of small-diameter connected subsets to approximate the  transfer operator by a finite rank operator, see \cite{liverani01}, references therein for an account. Because repeated application of the self-consistent transfer operator amounts to sequential application of distinct linear transfer operators, it is tempting to try to apply Ulam's method in a sequential way, creating finite-rank estimates of each linear transfer operator in a long sequence and multiplying together against an initial probability density. Such an approach has been proven to converge to quenched random invariant densities of cocycles of general random sequences of maps \cite{FGQ14,FGM19}. 

In the setting of self-consistent transfer operators there are two main obstacles to this approach.
The first is that  
one needs a discretisation that preserves two regular Banach spaces where the (weak) derivative in the more regular space belongs to the less regular one, such as the Sobolev spaces $W^{1,1}$ and $W^{2,1}$. Obviously a discretisation based on Ulam's scheme would fail
because the indicators that form the Ulam basis do not lie in $W^{1,1}$.
This regularity requirement creates a second issue, namely that the finite-rank operator that we use to map functions from our Banach space to a computer-accessible finite-dimensional space should be norm-reducing in the Banach spaces so that the contraction properties encoded in Lasota--Yorke inequalities are not disturbed.
Convolution with F\'ejer kernels \cite{Fe1904} overcomes both of these obstacles.
Order-$n$ F\'ejer kernel convolution maps functions in $W^{k,1}$, $k\ge 1$ to linear combinations of Fourier modes of order at most $n-1$, and is a weak contraction in the $W^{k,1}$ norm (unlike the sharper Dirichlet kernel \cite{W19}, which has other advantages in different settings).
F\'ejer kernel discretisation has been used to study perturbations and approximations of linear transfer operators \cite{CF20,FGQ14,FP25}, and we show that this scheme is particularly well-suited to nonlinear self-consistent transfer operators.
Using F\'ejer kernel convolution we form an \textit{approximate nonlinear } self-consistent transfer operator as a finite-rank map onto a finite-dimensional approximation space spanned by Fourier modes.
 In Theorem \ref{thm:main1} we prove that the approximate self-consistent transfer operator has a fixed point, $h^*_{\eps,N}$, and that this fixed point converges to the unique fixed point of the original self-consistent transfer operator, $h^*_\eps$ with a rate $\log{N}/N$, where $N$ is the maximal order of the Fourier modes forming the approximation space, and $\eps$ is the coupling strength in the underlying mean-field system. 
 Having established the theoretical efficacy of our approximation we turn to its numerical realisation. When the true fixed point of the self-consistent transfer operator is unique, one has the option of sequentially applying approximate linear transfer operators to effect a ``power method'' of estimating the fixed point of the approximate self-consistent transfer operator. In Theorem \ref{thm:seqiter} we show that such an iteration scheme converges exponentially quickly; this is proved in Section \ref{sec:thm21proof}.

 The self-consistent transfer operator is a nonlinear operator and it is therefore natural to develop a fixed point scheme for nonlinear operators.
 Newton's method is a good candidate due to its known rapid convergence and its ability to find both attracting and repelling fixed points. Developing the required theoretical underpinnings of Newton's method in our context is the main goal of Section \ref{sec:newtheory}. In Proposition \ref{prop:deriv} we formulate an explicit expression for the Fr\'echet derivative of the approximate self-consistent transfer operator, and in Theorem \ref{newtonthm} we check the necessary regularity and resolvent-bound requirements to run Newton's method \cite{bartle55, hinze2008optimization} on the self-consistent transfer operator. This leads to a theoretical superexponential (quadratic or order-2) convergence rate toward the fixed point of the approximate self-consistent transfer operator in Newton iteration
 count.
 In Section \ref{sec:num} we exploit the explicit representation of the Fr\'echet derivative developed in Section \ref{sec:newtheory} to construct a novel and efficient implementation of the Newton scheme using our F\'ejer-based Fourier approximation. We also include theoretical consideration of the numerical errors incurred. 
 Finally in Section \ref{sec:examples} we apply our numerical schemes for sequential iteration and Newton's method to two dynamical systems with distinct couplings. We observe that Newton's method significantly outperforms sequential iteration due to a much faster convergence rate at the expense of only a modest increase in computation time. 
 These experiments confirm  Theorems \ref{thm:seqiter} and \ref{thm:newtiter} with exponential convergence rates observed for both schemes.
We also validate the theoretical result of $h_{\eps,N}^*\to h_{\eps}^*$ with rate $\log N/N$ (Theorem \ref{thm:main1}).
Our work opens the door to efficient and reliable numerical exploration of coupling effects on 
 fixed points of the nonlinear self-consistent transfer operator,
and thus to the influence of couplings on the invariant states of the associated infinite-dimensional mean-field systems.

\label{sec:intro}
\section{Summary of the main results}\label{sec:mainresult}
Following the introduction of the self-consistent transfer operator and known relevant theory in \S\ref{sec:setting}, we summarise our approximation scheme and state our main convergence results in \S\ref{sec:disc}.

\subsection{The self-consistent transfer operator}
\label{sec:setting}
Let $T:\mathbb S^1\to\mathbb S^1$ be a $C^5$ uniformly expanding circle map. 
Let $W^{i,1}$, $i=1,2, 3,4,$ denote the Sobolev space equipped with usual norm: for $f\in W^{i,1}$, $\|f\|_{1,i}=\sum_{j=1}^i\|f^{(j)}\|_1+\|f\|_1<\infty.$ We sometimes in the text below abuse notation and write $W^{0,1}$ instead of $L^1$.
Define $\mathcal{D}_1=\{f\in L^1:  f\ge 0, \int_{\mathbb{S}^1} f=1\}$. Similarly denote $\mathcal{D}_{i,1}=\mathcal{D}_1\cap W^{i,1}$ for $i=1,2$.
Define $V_0 = \{f\in L^1: \int f = 0\}$ and $V^{i,1}=W^{i,1}\cap V_0$ for $i=1,2$.

For $\eps>0$, $g\in C^5$, $f\in\mathcal{D}_1$ consider the function 
\begin{equation}\label{eq:couplingmap}
I_{\eps,f}(x)=x+\eps \int g(x,y)f(y)dy,
\end{equation}
and set
\begin{equation}\label{eq:coupleddy}
T_{\eps,f}=I_{\eps,f}\circ T.
\end{equation}
$T_{\eps,f}$ is understood as a coupled system representing infinitely many particles interacting via their mean field, the local dynamics of each particle is described by $T$, $\eps$ is the coupling strength and $f$ is the state of the system. For more information about such concepts see \cite{keller00} and the review article \cite{tanzi23}.

Let $\mathcal L_T:L^1\to L^1$ denote the transfer (Perron-Frobenius) operator associated with $T$; i.e., for $\psi\in L^1$, $\varphi\in L^{\infty}$
$$\int_{\mathbb S^1}\varphi\circ T\cdot \psi dx=\int_{\mathbb S^1}\varphi\cdot \mathcal L_T\psi dx.$$
 It is well known that when acting on $W^{i,1}$ $\mathcal L_T$ admits a Lasota-Yorke inequality and a spectral gap. Consequently, $T$ admits a unique absolutely continuous invariant measure whose density is in $W^{4,1}$ (See \cite{liverani01} subsection 10.2).

Let $\cM_1(\mathbb S^1)$ be the set of probability density functions on $\mathbb S^1$. For $i=0,1,2$, consider the \emph{self-consistent transfer operator} $\tilde\L_{\eps}: W^{i,1}\cap \cM_1(\mathbb S^1) \to \cM_1(\mathbb S^1)$ which is defined by
\begin{equation}\label{eq:strans}
\tilde\L_{\eps}(h)=\L_{T_{\eps,h}}h.
\end{equation}
for $n=0,1,2,\dots$, we also write 
\begin{equation}\label{eq:evolve}
h^{n+1}=\L_{T_{\eps,h^{n}}}h^n.
\end{equation}
Note that fixed points $h^*$ of the operator defined in \eqref{eq:strans} corresponds to absolutely continuous invariant measures with density $h^*$ of the coupled system defined in \eqref{eq:coupleddy}.

Let $\eps_0^*>0$ be small, choose $\lambda\in(0,1)$ and $C_0>0$ so that 
$$\mathcal T(\eps_0^*, C_0)=\{T_{\eps,f}: \forall \eps\in[0,\eps_0^*], \forall f\in\mathcal{D}_1, \sup_x\frac{1}{|T_{\eps,f}'(x)|}\le \lambda<1 \text{ and } \|T_{\eps,f}\|_{C^{5}}<C_0\, \}$$
is nonempty.
This makes $\cT:=\mathcal T(\eps_0^*, C_0)$ a family of uniformly `nearby', uniformly expanding maps that satisfy uniform (independent of $\eps, f$ and $g$) Lasota-Yorke inequalities on $W^{i,1}$ (See \cite{liverani01} subsection 10.2 for a proof of such inequalities):
$\exists C_{LY}>0$ such that for all $T_{\eps,f}\in\cT$, $h\in W^{i,1}$, $n\in\mathbb N$, $i=1,2,3,4$, 
\begin{equation}\label{eq:standLY}
\|\L_{T_{\eps,f}}^n h\|_{W^{i,1}}\leq \lambda^{in}\|h\|_{W^{i,1}}+ C_{LY}\|h\|_{W^{i-1,1}},
\end{equation}
and consequently $\L_{T_{\eps,f}}$ admits a uniform-in-$\eps$ spectral gap when acting on $W^{i,1}$, $i=1,\dots,4$.
We fix throughout the constants $\eps_0^*, \lambda$, and $C_0$ and also assume throughout that all maps $T_{\eps,f}$ lie in $\mathcal{T}$.
Let $K\ge \frac{C_{LY}}{1-\lambda^2}$ be fixed throughout, and define a fixed ball $$B_K=\{h\in \mathcal{D}_1\cap W^{2,1}: \|h\|_{W^{2,1}}\le K\}$$
in which all initial densities lie, in which we will show all iterated densities and fixed points lie.

It is well known (see Theorem \ref{thm:gen}) that for $\eps_0^*$ \emph{small enough}, for each $\eps\in[0,\eps_0^*]$, $\tilde\L_\eps$ admits a unique fixed point $h_\eps^*\in B_K$;  this property defines $\eps^*_0$. In this work, we are concerned with obtaining a numerical approximation of $h_\eps^*$. Our numerical approximation relies on the following discretisation scheme.
\subsection{Approximation scheme and convergence results} 
\label{sec:disc}
Denote the standard Fourier modes on the circle by $e_k(x):=e^{2\pi ikx}$.
For $N\in \mathbb N$ 
the Fej\'er kernel \cite{Fe1904} is defined by $K_N = \sum_{|k|\le N-1}(1-|k|/N)e_k$.
There are finitely many nonzero Fourier coefficients of $K_N$ and the operator convolving a function with $K_N$ is finite rank given by: for $f\in L^1$
\begin{equation}\label{eq:discrete}
    \Pi_Nf(x)=\int_{\mathbb S^1}K_N(x-y)f(y)dy.
\end{equation}
The finite-rank operator $\tilde{\L}_{\eps,N}$ is defined by $\tilde{\L}_{\eps,N}:=\Pi_N\tilde\L_\eps$.
Define 
\begin{equation}\label{eq:Fspace}
    F_N:=\mathrm{span}\{e_k: k=-N+1,-N+2,\ldots,N\}.
\end{equation}
For an initial $h^0_N\in F_N$, 
\begin{equation*}
\Pi_N\tilde\L_\eps h^0_N=\Pi_N\L_{T_{\eps,h^0_N}}h^0_N,
\end{equation*}
and we iterate this finite-rank operation as follows:
\begin{equation}
    \label{eq:discreteiter}
h^{n}_N=\Pi_N\mathcal{L}_{T_{h^{n-1}_N}}h^{n-1}_N, \qquad n=1,2,\ldots
\end{equation}

Recall that for $\varepsilon\in [0,\varepsilon_0^*]$ the operator $\tilde{\mathcal{L}}_\varepsilon$ has at least one fixed point.
The following theorem shows that $\Pi_N\tilde{\mathcal{L}}_\varepsilon$ has a fixed point $h_{\eps,N}^*$, and in the situation where $h^*_\eps$ is unique, $h_{\eps,N}$ converges to $h^*_\eps$ in the large-$N$ limit.
\subsubsection{Existence of a discrete fixed point and its convergence}
\begin{theorem}\label{thm:main1} 
\,

\begin{enumerate}
\item For $\eps\in[0,\eps_0^*]$ the operator $\Pi_N\tilde\L_\eps$ has a fixed point $h_{\eps,N}^*\in B_K$; 
\item There is $\eps_3^*\in (0,\eps_2^*]$, where $\eps_2^*$ is defined in Theorem \ref{thm:gen}, such that for $\eps\in[0,\eps_3^*]$ if $h_\varepsilon^*$ is the unique fixed point of $\tilde\L_\eps$ and $h_{\varepsilon,N}^*$ is a fixed point of  $\Pi_N\tilde\L_\eps$ then
$$\|h_{\varepsilon,N}^*-h_\eps^*\|_{W^{1,1}}=O\left(\frac{\ln N}{N}\right).$$
\end{enumerate}
\end{theorem}
\begin{proof}
 See Section \ref{sec:thm21proof}
\end{proof}
\subsubsection{Sequential iteration}

We now describe our first approach to numerically approximating $h_{\eps,N}^*$, which is based on sequentially iterating the sequence of linear transfer operators that appear in the self-consistent transfer operator iteration.
\begin{theorem}
\label{thm:seqiter}
There exists $\eps^*_4>0$ such that  for  $\eps\in[0,\eps_4^*]$ and sufficiently large $N$ (independent of $\eps$) the operator $\tilde\L_{\eps,N}$ has a unique fixed point $h^*_{\eps,N}\in B_K$.
Further, there exists $C_K>0$ and $\gamma\in (0,1)$ such that   for all $h^0\in B_K$, we have
\[
\|{\tilde\L}_{\eps,N}^{n}(h^0)-h_{\eps,N}^*\|_{W^{1,1}}\le C_K \gamma^n;
\]
i.e.\
\begin{equation}
    \label{eq:seqiter}
\|\Pi_N\mathcal{L}_{T_{\eps, h^{n-1}_N}}\circ\cdots\circ \Pi_N\mathcal{L}_{T_{\eps, h^{1}_N}}\circ\Pi_N\mathcal{L}_{T_{\eps, h^{0}_N}}h^{0} - h^*_{\eps,N}\|_{W^{1,1}}\le  C_K \gamma^n.
\end{equation}
\end{theorem}
\begin{proof}
    See Section \ref{sec:uniquedisc}.
\end{proof}
In practice, one could initialise \eqref{eq:seqiter} with $h^0=h_{0,N}^*$, the fixed point of the transfer operator $\Pi_N\mathcal{L}_T$, obtained from the leading eigenvector of this linear operator.
Theorem \ref{thm:seqiter} guarantees an exponential convergence rate to the discrete fixed point $h^*_{\eps,N}$.

\subsubsection{Newton's method}

Newton's method promises superexponential convergence rates to $h_N^*$ in iteration count $n$, rather than the exponential convergence in $n$ for sequential iteration provided by Theorem \ref{thm:seqiter}.
Sequential iteration requires the numerical construction of a new transfer operator at each iteration, so reducing the number of iterations would be advantageous.

To set up the Newton approach, denote $\mathcal{N}h:=\Pi_N-\tilde\L_{\eps,N}$, and note that $\mathcal{N}h_N^*=0$;  that is, $h_N^*$ is a zero of the nonlinear operator $\mathcal{N}$.
For an initial $h\in F_N$, the Newton update 
is 
$h-(D_h\mathcal{N})^{-1}\circ \mathcal{N}h$, where $D_f\mathcal{N}$ denotes the Fr\'echet derivative of $\mathcal N$ at $f$. 
Rewriting this explicitly in terms of our transfer operators, one has the update rule
\begin{equation}
    \label{newtonstep}
h_{\eps,N}^{n+1}:=h_{\eps,N}^{n} - \left(\Pi_N - D_{h_{\eps,N}^{n}}\tilde\L_{\eps,N}\right)^{-1}\circ\left(\Pi_N-\tilde\L_{\eps,N}\right)h_{\eps,N}^{n}, \qquad n=0,1,2,\ldots
\end{equation}
Newton requires an initialisation and for simplicity, we choose $h^0_{\eps,N}=h_{0,N}$, the fixed point of $\tilde\L_{0,N}=\Pi_N\mathcal{L}_T$, which is obtained from the leading eigenvector of the transfer operator $\Pi_N\mathcal{L}_T$.

\begin{theorem}
\label{thm:newtiter}
There is $\eps_7^*>0$ (see Theorem \ref{newtonthm}) such that for $\eps<\min\{ \eps^*_2,\eps^*_7\}$,
there is a $W^{1,1}$ neighbourhood $B_{K,N}'\subset B_K$ of $h_{\eps,N}^*$ such that for sufficiently large $N$ and all $h^0_{\eps,N}\in B_{K,N}'$, the Newton iteration \eqref{newtonstep} has order-$2$ convergence toward $h_{\eps,N}^*$, that is there is $C>0$ and $0<\tilde \gamma<1$, such that 
    \begin{equation*}
    \|h^n_{\eps,N}-h_{\eps,N}^*\|_{1,1}\le C\tilde \gamma^{2^n}.
    \end{equation*}
\end{theorem}
\begin{proof}
See the proof of Theorem \ref{newtonthm}.
\end{proof}

\section{Proof of existence and convergence of the discretised fixed points of $\tilde{\L}_{\eps,N}$}
\subsection{Proof of Part 1 of Theorem \ref{thm:main1}}
\label{sec:thm21proof}
We start by proving that $\Pi_N$ has good regularity properties and it provides a good approximation. This result may be considered classical (see \cite{Zyg03}[ Ch.~II, \S 2.4--2.6] for the main ingredients on trigonometric estimates) and its proof is provided for completeness.
\begin{lemma}\label{lemkern}
The following statements hold:
\begin{itemize}
\item[1)] for $f \in W^{i,1}$, $i=0,1,\dots$, $\|\Pi_Nf\|_{W^{i,1}}\leq \|f\|_{W^{i,1}}$;
\item[2)] $\exists C>0$ such that for $f \in W^{i+1,1}$, $i=0,1,\dots$, $$\|\Pi_N f-f\|_{W^{i,1}}\leq C\frac{\ln(N)}{N}\|f\|_{W^{i+1,1}}.$$
\end{itemize} 
\end{lemma}
\begin{proof}
For the first item, note that
$\|\Pi_N f\|_{L^1}\leq \|K_N\|_{L^1}\|f\|_{L^1}$ and $\|K_N\|_{L^1}= 1$. This covers the $i=0$ case. To prove item 1) for $i\ge 1$, notice that by definition for $f\in L^1$, $\Pi_N f \in C^{\infty}$. Therefore, for $f\in W^{i,1}$, $i=1,2,\dots$
\begin{align*}
(\Pi_N f)^{(i)}(x):=&\left(\int K_N(x-y)f(y)dy \right)^{(i)}\\
=&\int K_n^{(i)}(x-y)f(y)dy\\
=&\int K_n(x-y)f^{(i)}(y)dy\\
=&\Pi_N f^{(i)}(x).
\end{align*} 
Then the second statement follows by using the above and the first statement of the Lemma. To prove the last item, it is sufficient to prove  for $f\in W^{1,1}$,
\begin{equation}\label{eq:idclose}
\|\Pi_N f-f\|_{L^1}\leq A\frac{\ln(N)}{N} \|f'\|_{L^1}.
\end{equation} 
Also the proof of the inequality in the case of $W^{i+1,1}$, $i=1, \dots$ follows from \eqref{eq:idclose} and the fact that $(\Pi_N f)^{(i)}(x)=\Pi_N f^{(i)}(x)$. To prove \eqref{eq:idclose} we first change variables in the definition of $\Pi_n$ (see equation \eqref{eq:discrete}) and write: 
\begin{align*}
\|\Pi_N f-f\|_{L^1}&\leq \int \int |K_N (y)(f(x+y)-f(x))|dydx\\
&\leq \int \int K_n (y)\int |f'(u)|1_{[x,y+x]}(u)dudydx\\
&\leq \int \int |f'(u)|K_N (y)\int 1_{[x,y+x]}(u)dxdydu\\
&\leq \int \int |f'(u)|K_N (y)\int 1_{[0,y]}(u-x)dxdydu\\
&\leq \int \int |f'(u)|K_N (y)ydydu\\
&\leq \|f'\|_{L^1}\|yK_N (y)\|_{L^1}.
\end{align*}
To compute $\|yK_N (y)\|_{L^1}$, we recall the following expression of $K_n$ : for any $y\in \mathbb S^1$ with $\mathbb S^1$ identified with $[-\pi,\pi]$,
\begin{align*}
K_N(y):=\frac1N\left(\frac{\sin(Ny/2)}{\sin(y/2)}\right)^2.
\end{align*}
We will control it according to the value of $y$.
\paragraph{Case 1: $ |y|\geq 1$.}
One can control $yK_N(y)$ as follows,
\begin{align*}
    |yK_N(y)|\leq |y|\frac{ 1}{N\sin^{2}(\frac 1 2)}.
\end{align*}
\paragraph{Case 2: $\frac 1 {N^2}\le |y|< 1.$}
We use the following control of $\sin$,
 using $y\geq \sin(y)\geq y-\frac{y^3}{6}$, 
 \begin{align*}
 |yK_N(y)| \leq \frac{2}{N|y|(1-\frac{y^2}{6})^2}\leq \frac{3}{N|y|}    
 \end{align*}
\paragraph{Case 3: $|y|\leq \frac 1 {N^2}$.} 
 \begin{align*}
     |yK_N(y)| \leq |y|\frac{2N^2y^2}{Ny^2(1-\frac{y^2}{6})^2}\leq 8N|y|.
 \end{align*}
Therefore,
\begin{align*}
\int |yK_N(y)|dy \leq& \int_{1\geq |y|\geq \frac{1}{N^2}} \frac{8}{N|y|}dy\\
&+\int_{|y|\geq 1}y\frac{1}{N\sin^2(\frac 1 2)}dy+\int_{|y|\leq \frac{1}{N^2}} 8N|y|dy\\
\leq& \frac{16\ln(N)}{N}+\frac {16\pi}{N}+\frac{2\pi}{N\sin^2(\frac 1 2)}.    
\end{align*}
Thus, proving $\|yK_N (y)\|_{L^1}\leq A\frac{\ln(N)}{N}$ for some constant $A>0$ and the last item of the lemma.
\end{proof}
{
\begin{lemma}\label{lem:discLY_co}
For any $N\in\mathbb N$ the operator $\Pi_N\tilde\L_\eps$ satisfies the following:
\begin{enumerate}
\item For any $\eps\in[0,\eps^*_0]$, for all $h\in W^{i,1}$, $i=1,2,3,4$
\begin{equation}
\|\Pi_N\tilde\L_\eps h\|_{W^{i,1}}\leq \lambda^{i}\|h\|_{W^{i,1}}+ C_{LY}\|h\|_{W^{i-1,1}},\nonumber
\end{equation}
where $C_{LY}$ and $\lambda$ are the same constants as in \eqref{eq:seqLY}.
Moreover, exists $\tilde C_{LY}>0$ such that for and any sequence  $(f_j)_{j\in \mathbb{N}} 
\in(\mathcal D_1)^{\mathbb N}$,
$$
 \|\Pi_N\L_{T_{\eps,f_n}}\circ\cdots \circ \Pi_N\L_{T_{\eps, f_1}}h\|_{W^{i,1}}\le \lambda^{in}\|h\|_{W^{i,1}}+\tilde C_{LY}\|h\|_{W^{i-1,1}}.
$$
\item for $\eps\in[0,\eps_0^*]$, $\Pi_N\tilde\L_\eps$ leaves $B_K$ invariant and hence has a fixed point $h_N^*\in B_K$. 
\end{enumerate}
\end{lemma}
\begin{proof}
The first statement follows from 1), 2) of Lemma \ref{lemkern} and \eqref{eq:seqLY}. While the second statement follows from 2) of Lemma \ref{lemkern}, the fact that $\tilde\L_\eps(B_k)\subseteq B_k$ and the Schauder fixed point theorem.   
\end{proof}
This completes the proof of part 1 of Theorem \ref{thm:main1}
.}

\subsection{Proof of Part 2 of Theorem \ref{thm:main1}}

{
\begin{lemma}\label{lem:contracto}
There are $\eps_1^*>0$ and $N_0\in \mathbb N$, $\tilde \theta (N_0) >0$, such that for any $N\geq N_0$, $n\in\mathbb{N}$, $\eps\in[0,\eps_1^*]$, $h\in V^{i,1}$, $i=1,2$ and any sequence  $(f_j)_{j\in \mathbb{N}} \in (B_K)^{\mathbb{N}}$,
\begin{align*}
    \|\Pi_N\L_{T_{\eps,f_n}}\circ\cdots \circ \Pi_N\L_{T_{\eps, f_1}}h\|_{W^{i,1}}\le C \tilde \theta^n\|h\|_{W^{i,1}}.
\end{align*}
\end{lemma}
\begin{proof}
By Lemma \ref{lem:discLY_co}  $\exists C>0$ so that for any $h\in V^{i,1}$, $n\in \mathbb N$,
\begin{align}\label{eq:lunbound}
    \|\Pi_N\L_{T_{\eps,f_n}}\circ\cdots \circ \Pi_N\L_{T_{\eps, f_1}}h\|_{W^{i,1}}\leq C \|h\|_{W^{i,1}}.
\end{align}
Then,
\begin{align}
    &\|\Pi_N\L_{T_{\eps,f_n}}\circ\cdots \circ \Pi_N\L_{T_{\eps, f_1}}h-\L_{T_{\eps,f_n}}\circ\cdots \circ\L_{T_{\eps, f_1}}h\|_{W^{i-1, 1}}\nonumber\\
    &\leq \sum_{k=1}^{n}\|\Pi_N\L_{T_{\eps,f_n}}\circ\cdots \circ \Pi_N\L_{T_{\eps, f_{k+1}}}
    \left(\Pi_n-I\right)\L_{T_{\eps,f_{k}}}\circ\cdots \circ\L_{T_{\eps, f_1}}h\|_{W^{i-1, 1}}\nonumber\\
    &\leq C\sum_{k=1}^{n}\|(\Pi_N-I)\L_{T_{\eps,f_{k}}}\circ\cdots \circ\L_{T_{\eps, f_1}}h\|_{W^{i-1, 1}}\label{eq:memloss1}\\
    &\leq \sum_{k=1}^{n}C\theta^{k}\|\Pi_N-I\|_{W^{i,1}\to{W^{i-1,1}} }\|h\|_{W^{i,1}}\label{eq:memloss2}\\
    &\leq CA \frac{\ln(N)}{N}\sum_{k=1}^{n}\theta^{k}\|h\|_{W^{i,1}}\label{eq:memloss2'}\\
    &\leq C\frac{A\theta}{1-\theta} \frac{\ln(N)}{N}\|h\|_{W^{i,1}},\label{eq:memloss3}
\end{align}
where in \eqref{eq:memloss1} we used \eqref{eq:lunbound}, in \eqref{eq:memloss2} the second item of Theorem \ref{thm:gen}, and in \eqref{eq:memloss2'} the third item of Lemma \ref{lemkern}.
Now let $n,m\in \mathbb N$, using first Lemma \ref{lem:discLY_co}, followed by \eqref{eq:memloss3}, then the second item of Theorem \ref{thm:gen}, we obtain
 \begin{align*}
    &\|\Pi_N\L_{T_{\eps,f_{n+m}}}\circ\cdots \circ \Pi_N\L_{T_{\eps, f_1}}h\|_{W^{i,1}}\\
   &\le \lambda^{in}\|\Pi_N\L_{T_{\eps,f_{m}}}\circ\cdots \circ \Pi_N\L_{T_{\eps, f_1}}h\|_{W^{i,1}}
   +\tilde C_{LY}\|\Pi_N\L_{T_{\eps,f_{m}}}\circ\cdots \circ \Pi_N\L_{T_{\eps, f_1}}h\|_{W^{i-1,1}}\\
   &\le \lambda^{i(n+m)}\|h\|_{W^{i,1}} +C_1\frac{A\theta}{1-\theta} \frac{\ln(N)}{N}\|h\|_{W^{i,1}}+C_2\lambda^{im}\|h\|_{W^{i,1}}.
 \end{align*}
Consequently, choosing $N$ Large enough so that $C_1\frac{A\theta}{1-\theta} \frac{\ln(N)}{N}<\frac13\alpha$, then $m$ large enough so that $C_2\lambda^{im}<\frac13\alpha$, and finally $n$ large enough so that $\lambda^{i(n+m)}<\frac13\alpha$, for some $\alpha\in(0,1)$, to obtain 
$$\|\Pi_N\L_{T_{\eps,f_{n+m}}}\circ\cdots \circ \Pi_N\L_{T_{\eps, f_1}}h\|_{W^{i,1}}\le\alpha \|h\|_{W^{i,1}}.$$
The proof is completed by iterating the above inequality.
\end{proof}

\begin{lemma}\label{lem:convhN}
There is $\eps_{3}^*\in (0, \eps_2^*)$ such that for $\eps\in [0,\eps_{3}^*]$, if $h_\varepsilon^*$ is the unique fixed point of $\tilde\L_\eps$ and $h_{\varepsilon,N}^*$ is a fixed point of  $\Pi_N\tilde\L_\eps$ then
$$\|h_{\varepsilon,N}^*-h_\eps^*\|_{W^{1,1}}=O\left(\frac{\ln N}{N}\right).$$
\end{lemma}

\begin{proof}

For $\eps\in[0,\eps^*_{3}]$, for any $N\in\mathbb N$, by Lemma \ref{lem:discLY_co}, the operator $\Pi_N\tilde\L_\eps$ has a fixed point $h_{\eps,N}^*\in B_K$, and by Theorem \ref{thm:gen},
\[
\|{\tilde\L}_\eps^{n}(h_{\eps,N}^*)-h_\eps^*\|_{W^{1,1}}\le C_K \gamma^n,
\]
where $h_\eps^*\in B_K$ is the unique fixed of $\tilde\L_\eps$. We have
\begin{equation}\label{eq:fixpoint}
\begin{split}
\|h_{\varepsilon,N}^*-h_\eps^*\|_{W^{1,1}}\le& \|(\Pi_N\tilde\L_\eps)^n(h_{\varepsilon,N}^*)-{\tilde\L}_\eps^{n}(h_{\eps,N}^*)\|_{W^{1,1}}+\|{\tilde\L}_\eps^{n}(h_{\eps,N}^*)-h_\eps^*\|_{W^{1,1}}\\
&\le\|(\Pi_N\tilde\L_\eps)^n(h_{\varepsilon,N}^*)-{\tilde\L}_\eps^{n}(h_{\eps,N}^*)\|_{W^{1,1}} +C_K \gamma^n.
\end{split}
\end{equation}
Where $(\Pi_N\tilde\L_\eps)^n$ is a notation for the $n^{th}$ iterate of the map $\Pi_N\tilde\L_\eps$.  Define the sequence $(h_{\eps,N}^{*_{n}})_{n\in \mathbb N}$ by $h_{\eps,N}^{*_n}:=\tilde \L_\eps^n h_N^*$. Then,
\begin{align*}
&\|(\Pi_N\tilde\L_\eps)^n(h_{\varepsilon,N}^*)-{\tilde\L}_\eps^{n}(h_{\eps,N}^*)\|_{W^{1,1}}\\
&=\|\sum_{k=0}^{n-1}(\Pi_N\circ \L_{T_{\eps,h_{\varepsilon,N}^*}})^{n-k}\left(\Pi_N \L_{\eps,T_{\eps,h_N^*}}-\L_{T_{\eps,h_{\eps,N}^{*_k}}}\right) \L_{T_{\eps,h_{\eps,N}^{*_{k-1}}}}\circ\cdots \circ\L_{T_{\eps, h_{\eps,N}^{*}}}h_{\eps,N}^*\|_{W^{1, 1}}\nonumber\\
&=\|\sum_{k=0}^{n-1}(\Pi_N\circ \L_{T_{\eps,h_{\varepsilon,N}^*}})^{n-k}\left[\left(\Pi_N-I\right) \L_{\eps,T_{\eps,h_N^*}}+\left(\L_{\eps,T_{\eps,h_N^*}}-\L_{T_{\eps,h_{\eps,N}^{*_k}}}\right)\right] \L_{T_{\eps,h_{\eps,N}^{*_{k-1}}}}\circ\cdots \circ\L_{T_{\eps, h_{\eps,N}^{*}}}h_{\eps,N}^*\|_{W^{1, 1}}\nonumber\\
&=(I)+(II),
    \end{align*}
where 
$$(I):= \sum_{k=0}^{n-1}(\Pi_N\circ \L_{T_{\eps,h_{\varepsilon,N}^*}})^{n-k}\left[\left(\Pi_N-I\right) \L_{\eps,T_{\eps,h_N^*}}\right] \L_{T_{\eps,h_{\eps,N}^{*_{k-1}}}}\circ\cdots \circ\L_{T_{\eps, h_{\eps,N}^{*}}}h_{\eps,N}^*\|_{W^{1, 1}},$$
and 
$$
(II):=\|\sum_{k=0}^{n-1}(\Pi_N\circ \L_{T_{\eps,h_{\varepsilon,N}^*}})^{n-k}\left[\L_{\eps,T_{\eps,h_N^*}}-\L_{T_{\eps,h_{\eps,N}^{*_k}}}\right] \L_{T_{\eps,h_{\eps,N}^{*_{k-1}}}}\circ\cdots \circ\L_{T_{\eps, h_{\eps,N}^{*}}}h_{\eps,N}^*\|_{W^{1, 1}}.
$$
We first treat $(I)$. Note in $(I)$, $(\Pi_N\circ \L_{T_{\eps,h_{\varepsilon,N}^*}})^{n-k}$ acts on a zero-mean function, thus Lemma \ref{lem:contracto} applies and using Lemma \ref{lemkern} alongside it to control the term $(\Pi_N-I)$, we obtain:

\begin{align}
(I)&=\|\sum_{k=0}^{n-1}(\Pi_N\circ \L_{T_{\eps,h_{\varepsilon,N}^*}})^{n-k}\left[\left(\Pi_N-I\right) \L_{\eps,T_{\eps,h_N^*}}\right] \L_{T_{\eps,h_{\eps,N}^{*_{k-1}}}}\circ\cdots \circ\L_{T_{\eps, h_{\eps,N}^{*}}}h_{\eps,N}^*\|_{W^{1, 1}}\nonumber\\
&\leq \sum_{k=0}^{n-1}\tilde \theta^{n-k}\|\left(\Pi_N-I\right)\|_{W^{2,1}\to W^{1,1}} \|\L_{\eps,T_{\eps,h_N^*}}\circ \L_{T_{\eps,h_{\eps,N}^{*_{k-1}}}}\circ\cdots \circ\L_{T_{\eps, h_{\eps,N}^{*}}}h_{\eps,N}^*\|_{W^{2, 1}}\nonumber\\
&\leq C\frac{\ln(N)}{N}\sum_{k=0}^{n-1}\tilde \theta^{n-k} \|\L_{\eps,T_{\eps,h_N^*}}\circ \L_{T_{\eps,h_{\eps,N}^{*_{k-1}}}}\circ\cdots \circ\L_{T_{\eps, h_{\eps,N}^{*}}}h_{\eps,N}^*\|_{W^{2, 1}}\nonumber\\
&\leq KC\frac{\ln(N)}{N}\sum_{k=0}^{n-1}\tilde \theta^{n-k},\label{eq:fixpartun}
\end{align}
where in the last line we used that $h_N^*\in B_K$ and Lemma \ref{lem:discLY_co}.

We now deal with $(II)$, using again Lemma \ref{lem:contracto} and \eqref{eq:closeop} from Lemma \ref{introprop},
\begin{align}
(II)&=\|\sum_{k=0}^{n-1}(\Pi_N\circ \L_{T_{\eps,h_{\varepsilon,N}^*}})^{n-k}\left[\L_{\eps,T_{\eps,h_N^*}}-\L_{T_{\eps,h_{\eps,N}^{*_k}}}\right] \L_{T_{\eps,h_{\eps,N}^{*_{k-1}}}}\circ\cdots \circ\L_{T_{\eps, h_{\eps,N}^{*}}}h_{\eps,N}^*\|_{W^{1, 1}}\nonumber\\
&\leq \sum_{k=0}^{n-1}\tilde \theta^{n-k}\|\left[\L_{\eps,T_{\eps,h_N^*}}-\L_{T_{\eps,h_{\eps,N}^{*_k}}}\right] \L_{T_{\eps,h_{\eps,N}^{*_{k-1}}}}\circ\cdots \circ\L_{T_{\eps, h_{\eps,N}^{*}}}h_{\eps,N}^*\|_{W^{1, 1}}\nonumber\\
&\leq \sum_{k=0}^{n-1}\tilde \theta^{n-k} \eps \|h_N^*-h_{\eps,N}^{*_k}\|_{W^{1,1}} \|\L_{T_{\eps,h_{\eps,N}^{*_{k-1}}}}\circ\cdots \circ\L_{T_{\eps, h_{\eps,N}^{*}}}h_{\eps,N}^*\|_{W^{2, 1}}\nonumber\\
&\leq K\eps\sum_{k=1}^{n-1}\tilde \theta^{n-k}\|h_N^*-h_{\eps,N}^{*_k}\|_{W^{1,1}},\label{eq:fixpartdeux}
\end{align}
where the last line is obtained using the fact that $h_N^*\in B_K$ and \eqref{eq:seqLY}.
Let $$U_n^N:=\|(\Pi_N\tilde\L_\eps)^n(h_{\varepsilon,N}^*)-{\tilde\L}_\eps^{n}(h_{\eps,N}^*)\|_{W^{1,1}}.$$ 
Notice that by \eqref{eq:fixpartun} and \eqref{eq:fixpartdeux}, satisfied by $(I)$ and $(II)$, we obtain the following inequality

\begin{align*}
U_n^N &\leq KC\frac{\ln(N)}{N}\sum_{k=0}^{n-1}\tilde \theta^{n-k}+K\eps\sum_{k=0}^{n-1}\tilde \theta^{n-k}U_k^N\\
&\leq KC\frac{\ln(N)}{N}\frac{\tilde \theta}{1-\tilde \theta}+K\eps\sum_{k=0}^{n-1}\tilde \theta^{n-k}U_k^N.
\end{align*}
Therefore, for any $n\in \mathbb N$, $U_n^N\leq V_n^N$ where $(V_n^N)_{n\in \mathbb N}$ is a sequence defined inductively by

\begin{align*}
  V_n^N=KC\frac{\ln(N)}{N}\frac{\tilde \theta}{1-\tilde \theta}+K\eps\sum_{k=1}^{n-1}\tilde \theta^{n-k}V_k^N,
\end{align*}
with $V_1^N=U_1^N=\|h_N^*-h_{\eps,N}^{*_1}\|_{W^{1,1}}$. In other words, 
\begin{align*}
  V_n^N&=KC\frac{\ln(N)}{N}\frac{\tilde \theta}{1-\tilde \theta}+K\eps\sum_{k=1}^{n-1}\tilde \theta^{n-k}V_k^N\\
  &=KC\frac{\ln(N)}{N}\frac{\tilde \theta}{1-\tilde \theta}+K\eps \tilde \theta V_{n-1}^N +\tilde \theta K \eps\sum_{k=1}^{n-2}\tilde \theta^{n-1-k}V_k^N\\
  &=KC\frac{\ln(N)}{N}\frac{\tilde \theta}{1-\tilde \theta}+K\eps \tilde \theta V_{n-1}^N +\tilde \theta \left( V_{n-1}^N-KC\frac{\ln(N)}{N}\frac{\tilde \theta}{1-\tilde \theta}\right)\\
&=KC\frac{\ln(N)}{N}\tilde \theta+(K\eps +1)\tilde \theta  V_{n-1}^N.
\end{align*}
Consequently, choosing $\eps_{3}^*$ such that $(K\eps_{3}^* +1)\tilde \theta:=\tilde\theta_1<1$, for any $\eps <\eps_{3}^*$ the sequence $(V_n^N)_{n\in \mathbb N}$ is of the form
\begin{align*}
V_n^N:=\tilde\theta_1^{n-1}V_1^N+KC\frac{\ln(N)}{N}\frac{\tilde \theta}{1-\tilde \theta}.
\end{align*}
Thus, using the above in \eqref{eq:fixpoint}, for any $n\in \mathbb N$,

\begin{align*}
\|h_{\varepsilon,N}^*-h_\eps^*\|_{W^{1,1}}
&\leq C_K \gamma^n+\tilde\theta_1^{n-1}\|h_N^*-h_{\eps,N}^{*_1}\|_{W^{1,1}}+KC\frac{\ln(N)}{N}\frac{\tilde \theta}{1-\tilde \theta}.\\
\end{align*}
One can then choose $n=-\ln(N)(\ln(\gamma)\vee \ln(\tilde\theta_1))$ to obtain
\begin{align*}
\|h_{\varepsilon,N}^*-h_\eps^*\|_{W^{1,1}}&\leq  C_K N^{-1}+N^{-1}\|h_N^*-h_{\eps,N}^{*_1}\|_{W^{1,1}}+KC\frac{\ln(N)}{N}\frac{\tilde \theta}{1-\tilde \theta}\\
&\leq  C_K N^{-1}+N^{-1}KC\frac{\ln(N)}{N}+KC\frac{\ln(N)}{N}\frac{\tilde \theta}{1-\tilde \theta}.
\end{align*}
\end{proof}
This completes the proof of part 2 of Theorem \ref{thm:main1}.
\subsection{Proof of Theorem \ref{thm:seqiter}}\label{sec:uniquedisc}
\begin{lemma}\label{lem:discrete_contraction}
There exists $\eps_4^*\in(0,\eps_1^*)$ such that for $\eps\in [0,\eps_4^*]$, and $N$ large enough, the operator $\tilde\L_{\eps,N}$ has a unique fixed point $h^*_{\eps,N}\in B_K$. 
Moreover, there exists $\tilde{C}_K>0$ and $\gamma\in (0,1)$ such that for all $\eps \in[0,\eps^*_4]$,  for all $h^0\in B_K$, we have
\[
\|{\tilde\L}_{\eps,N}^{n}(h^0)-h_{\eps,N}^*\|_{W^{1,1}}\le C_K \gamma^n;
\]
i.e.
    $$
\|\Pi_N\mathcal{L}_{T_{\eps, h^{n-1}_N}}\circ\cdots\circ \Pi_N\mathcal{L}_{T_{\eps, h^{1}_N}}\circ\Pi_N\mathcal{L}_{T_{\eps, h^{0}_N}}h^{0} - h^*_{\eps,N}\|_{W^{1,1}}\le C_K\gamma^n.
$$
\end{lemma}

\begin{proof}
Let $h_{\eps,N}^n:= (\Pi_N\circ \tilde \L_\eps)^nh^1$ and $\tilde h_{\eps,N}^n:= (\Pi_N\circ \tilde \L_\eps)^n\tilde h^1$. Let $\tilde\theta$ and $C$ be as in Lemma \ref{lem:contracto}. Choose $\gamma \in (\tilde\theta,1)$ and $C_0 > \max\{1,C\}$. We prove by induction that
$\|h_{\eps,N}^n-\tilde h_{\eps,N}^n\|_{W^{1,1}}\le C_0\gamma^n$. To do this assume 
$$\|h_{\eps,N}^k-\tilde h_{\eps,N}^k\|_{W^{1,1}}\le C_0\gamma^k\|h_{\eps,N}-\tilde h_{\eps,N}\|_{W^{1,1}}$$
for any $k=1,\dots, n-1$ and prove it for $n$. Using Lemma  \ref{lem:contracto}, the first item of Lemma \ref{lemkern}, the property \eqref{eq:closeop} and the induction hypothesis, we have
\begin{equation*}
\begin{split}
   &\|(\Pi_N\circ \tilde \L_\eps)^nh^1-(\Pi_N\circ \tilde \L_\eps)^n\tilde h^1\|_{W^{1,1}}\\
   &\leq \| \left((\Pi_N\circ \L_{T_{\eps, h_{\eps,N}^n}})\dots (\Pi_N\circ \L_{T_{\eps,h^1}})-(\Pi_N\circ \L_{T_{\eps, \tilde h_{\eps,N}^n}})\dots (\Pi_N\circ \L_{T_{\eps, \tilde h^1}})\right)h^1\|_{W^{1,1}}\\
   &+\|(\Pi_N\circ \L_{T_{\eps, \tilde h_{\eps,N}^n}})\dots (\Pi_N\circ \L_{T_{\eps, \tilde h^1}}(h^1-\tilde h^1)\|_{W^{1,1}}\\
    &\le C_0CK\eps\sum_{k=1}^n \tilde\theta^{n-k}\gamma^k \|h_{\eps,N}-\tilde h_{\eps,N}\|_{W^{1,1}}+ C\tilde\theta^n\|h_{\eps,N}-\tilde h_{\eps,N}\|_{W^{1,1}}\\
    &\le C_0CK\gamma^n\eps\sum_{k=1}^n (\tilde\theta/\gamma)^{n-k}\|h_{\eps,N}-\tilde h_{\eps,N}\|_{W^{1,1}}+ C\tilde\theta^n\|h_{\eps,N}-\tilde h_{\eps,N}\|_{W^{1,1}}\\
    &\le C_0\gamma^n\|h_{\eps,N}-\tilde h_{\eps,N}\|_{W^{1,1}}
\end{split}
\end{equation*}
provided $\eps$ is small enough.
\end{proof}
This completes the proof of part 2 of Theorem \ref{thm:seqiter}.
\section{Fixed point approximation by Newton iteration}\label{sec:newtheory}
In this section we develop a version of Newton's method that will be used for our numerical scheme in order to estimate a fixed point $h_N^*$. Newton's method requires the existence of, and an explicit expression for, the Fr\'echet derivative for the discrete operator $\tilde \L_{\eps,N}$.
We note that the existence of a Fr\'echet derivative for the true self-consistent operator $\tilde \L_\eps$ in a neighbourhood of the fixed point has been shown in \cite{CGT25}.
The development of an explicit expression for the derivative of the discretised self-consistent transfer operator is the focus of Section \ref{subsecdiff}, with an explicit formula given in Proposition \ref{prop:deriv}.

In order to clarify the general requirements for Newton's method on Banach spaces, we assume the existence of a Fr\'echet derivative for $\tilde \L_{\eps,N}$ and for brevity denote $h_N^*:=h_{\eps,N}^*$. Let
$\mathcal{N}h:=\Pi_N-\tilde\L_{\eps,N}$, and note that $\mathcal{N}h_N^*=0$;  that is, $h_N^*$ is a zero of the nonlinear operator $\mathcal{N}$. We follow the approach of \cite{hinze2008optimization}.
Recall the definition of the finite dimensional space $F_N$ introduced in \eqref{eq:Fspace}. For an initial $h^0\in F_N$, the Newton update is 
\begin{equation}
    \label{eq:newtonupdate}
h^{n+1}:=h^n-(D_{h^n}\mathcal{N})^{-1}\circ \mathcal{N}h^n,
\end{equation}
where $D_f\mathcal{N}$ denotes the Fr\'echet derivative of $\mathcal N$ at $f$. 

By \eqref{eq:newtonupdate} and the fact that $h^*_N$ is a zero of $\mathcal{N}$ one has 
\begin{align*}
D_{h^n}\mathcal{N}(  
    h^*_N-h^{n+1})&=D_{h^n}\mathcal{N}(h^*_N) -(D_{h^n}\mathcal{N}(h^n-(D_{h^n}\mathcal{N})^{-1}\circ \mathcal{N}h^n))-\mathcal{N}h^*_N\nonumber\\
    &=D_{h^n}\mathcal{N}(h^*_N - h^n) +\mathcal{N}h^n-\mathcal{N}h^*_N.
\end{align*}
Thus, 
\begin{equation*}
    h^*_N-h^{n+1} =(D_{h^n}\mathcal{N})^{-1}\left(D_{h^n}\mathcal{N}(h^*_N-h^n)+\mathcal{N}(h^*_N+e^n)-\mathcal{N}(h^*_N)\right),
\end{equation*}
where $e^n:=h^n-h_N^*$ is the error at step $n$.
Therefore if one has the inequality
\begin{align*}
\|e^{n+1}\|=\|(D_{h^n}\mathcal{N})^{-1}\left(\mathcal{N}(h^*_N+e^n)-\mathcal{N}(h^*_N)- D_{h^n}\mathcal{N}(e^n)\right)\|\le C\|e^n\|^{1+\alpha},
\end{align*}
one achieves order $1+\alpha$ convergence when $h^n$ is sufficiently close to $h_N^*$.
This condition is usually stated as:
\begin{enumerate}
    \item[a)] For some $M<\infty$ one has the following bound 
    \begin{equation}
        \label{eq:resolventnewton}
    \|(D_{h^n}\mathcal{N})^{-1}\|<M
    \end{equation}  along the Newton sequence $h^n$,
    \item[b)] For some $C<\infty$, the Fr\'echet derivative satisfies 
    \begin{equation}
        \label{eq:Frechetcondition}
    \|\mathcal{N}(h^*_N+e^n)-\mathcal{N}(h^*_N)- D_{h^*_N+e^n}\mathcal{N}(e^n)\|\le C\|e^n\|^{1+\alpha}.
    \end{equation}
\end{enumerate}
Our approach will be to demonstrate condition (a) uniformly in a neighbourhood of $h_N^*$.
Condition (b) may be reduced to a condition on the Holder regularity of $e \mapsto D_{h^*_N+e}\mathcal{N}(e)$ :
\begin{align*} 
    \|\mathcal{N}(h^*_N+e^n)-\mathcal{N}(h^*_N)- D_{h^*_N+e^n}\mathcal{N}(e^n)\|& \le \left\|\int_0^1D_{h^*_N+te^n}\mathcal{N}(e^n)\ dt- D_{h^*_N+e^n}\mathcal{N}(e^n)\right\|\nonumber\\
    &\leq \int_0^1\|D_{h^*_N+te^n}\mathcal{N}(e^n)- D_{h^*_N+e^n}\mathcal{N}(e^n)\|\ dt.
    \end{align*}
Thus to obtain condition (b), it is sufficient to ask for 
 \begin{align}\label{eq:Frechetderivreg}
\|D_{h^*_{\eps,N}+te}\tilde{\L}_{\eps,N}(e)- D_{h^*_{\eps,N}}\tilde{\L}_{\eps,N}(e)\|\le C\|e\|^{1+\alpha}, \,\,  \forall e\in F_N, \, 0\le t\le 1.    
 \end{align}
These considerations lead to Theorem \ref{newtonthm} below.

We introduce the  space $V^{1,1}:=\{h\in W^{1,1}, \int h=0\}$ and for $N\in \mathbb{N}$ we fix the constant $C_N^{(2)}>1$ to be such that for any $e\in W^{2,1}\cap F_N$,
\begin{align}\label{eq:equivcon}
\|e\|_{W^{2,1}}\leq C_N^{(2)}\|e\|_{W^{1,1}}.    
\end{align}

\begin{theorem}
  \label{newtonthm}
There is $\eps_7^*>0$ such that for fixed $\eps< \text{min}\{\eps_2^*,\eps_7^*\}$, large enough $N$, and $K> \frac{C_{LY}}{1-\lambda^2}$, we let $ \Lambda_{K,N}>\sup_{f\in B_K\cap F_N}\|(\Pi_N-D_f\tilde{\L}_{\eps,N})^{-1}|_{F_N}\|_{V^{1,1}}$, $\delta_{K,N}:=\sup\{\delta: \forall f,g \in F_N\cap \mathcal{D}_1\cap B_K,  \|f-g\|_{W^{1,1}}\leq \delta \Rightarrow  \|D_f\tilde\L_{\eps,N}-D_g\tilde\L_{\eps,N}\|_{W^{1,1}}\leq \frac{1}{4\Lambda_{K,N}}\}$
and $\beta_{K,N}\leq \min\{1,\delta_{K,N}\}$. 
Then there is $\tilde K >0$ such that the set $ B_{\tilde K,N}':=\{h  :\mathcal D_1\cap F_N\cap B_{\|\|_{W^{1,1}}}(h,\beta_{\tilde K,N})\subset B_{\tilde K}, \|(\Pi_N-\tilde\L_{\eps,N})h\|_{W^{1,1}}< \frac{\beta_{\tilde K, N}}{2\Lambda_{\tilde K,N}}\}$ contains $h_{\eps,N}^*$ and 
the Newton scheme \eqref{eq:newtonupdate} initialised with any  $h^0_{\eps,N}\in B_{\tilde K,N}'\cap B_{\|\|_{W^{1,1}}}(h_{\eps,N}^*,\beta_{\tilde K,N})$  converges to the fixed point $h_{\eps,N}^*$  with order-$2$ convergence in $\|\cdot\|_{W^{1,1}}$: i.e.\ there is $C>0$ and $0<\tilde \gamma<1$, such that for all $n\ge 0$,
   \begin{align}
   \label{newtonrate}
    \|h^n_{\eps,N}-h_{\eps,N}^*\|_{1,1}\le C\tilde \gamma^{2^n}.
    \end{align}
\end{theorem}

\begin{proof}

{We first prove convergence of Newton's method for initialisations sufficiently close to the fixed point, and provide the rate in \eqref{newtonrate}. After that we specify more precisely the neighbourhood size to guarantee convergence.
The proof of the order-2 rate} of convergence relies on a series of results in Section \ref{sec:newtheory}. Proposition \ref{prop:deriv} shows that $f\in F_N\mapsto  \tilde \L_{\eps,N}f \in F_N$ admits a Fr\'echet derivative {in $W^{1,1}$}.
The {local} convergence of Newton then relies on conditions a) and b) stated above and applied with the norm $\|\cdot\|_{W^{1,1}}$. Condition a), that is the boundedness of the resolvent of the Fr\'echet derivative in equation \eqref{eq:resolventnewton}, is a consequence of Proposition \ref{prop:4.4}. As for condition b), as stated above, it is sufficient to demonstrate \eqref{eq:Frechetderivreg} which is a consequence, when $\alpha=1$, of Lemma \ref{lemcontinuityderiv}.
These conditions a) and b) then imply, according to {Part 3 of} Theorem 2.9 \cite{hinze2008optimization}, the convergence of the Newton scheme {at a rate of order $1+\alpha$ for initial conditions sufficiently close to the fixed point $h^*_{\eps,N}$}.
In particular, there is $C_{0,N}>0$ such that, for $e^n:=h^n_{\eps,N}-h_{\eps,N}^*$, there is $n_0\in \mathbb{N}$ such that for any $n\geq n_0$,
\begin{align*}
\|e^{n+1}\|_{W^{1,1}}\le C_{0,N}\|e^n\|_{W^{1,1}}^{2}.
\end{align*}

Assume {for the beginning of an induction argument, starting at $n=n_0'$, that} there is $n_0'\ge n_0$  so that $ \gamma_0:=C_{0,N}\|e^{n_0'}\|_{W^{1,1}}<1$. 
{Then defining $\tilde \gamma:=\gamma_0^{2^{-n_0'}}$ one has the desired inequality at $n_0'$, namely} $\|e^{n_0'}\|_{W^{1,1}}\leq C_{0,N}^{-1}\tilde \gamma^{2^{n_0'}}$.
{We now check the induction step.} 
Assume $\|e^{n}\|_{W^{1,1}}\le C_{0,N}^{-1}\tilde \gamma^{2^n}$ for some $n\geq n_0'$
\begin{align*}
    \|e^{n+1}\|_{W^{1,1}}&\leq C_{0,N}\|e^n\|_{W^{1,1}}^{2}\\
    &\le  C_{0,N}\left(C_{0,N}^{-1}\tilde \gamma^{2^n}\right)^2\\
     &\le  C_{0,N}^{-1}\tilde \gamma^{2^{n+1}}.
\end{align*}
This proves by induction that the Newton method converges in a {sufficiently small} neighbourhood of $h_{\eps,N}^*$ and the rate of convergence is of order $2$.
{This may be equivalently phrased as:} there is $C>0$ and $0<\tilde \gamma<1$ such that for any $e^0$ such that $C_{0,N}\|e^{0}\|_{W^{1,1}}<1$,
\begin{align}\label{eq:speedconv}
    \|e^n\|_{W^{1,1}}\le C\tilde \gamma^{2^n}.
\end{align}
{To construct the {neighbourhood about the fixed point} where the Newton scheme converges recall that for any $K> \frac{C_{LY}}{1-\lambda^2}$, $B_K$ is invariant under $\tilde\L_{\eps,N}$ and $h_{\eps,N}^*\in B_K$.

Furthermore by construction, $\beta_{K,N}$ satisfies the same criterion as the constant $\beta$ from the proof of \citep[Lemma 2]{bartle55}, that is for any $f_0 \in B_{K}$ such that $\mathcal D_1\cap F_N\cap B_{\|\cdot\|_{W^{1,1}}}(f_0,\beta_{K,N})\subset B_{K}$,
\begin{enumerate}
    \item  $\|(\Pi_N-D_f\tilde{\L}_{\eps,N})^{-1}|_{F_N}\|_{V^{1,1}}< \Lambda_{K,N}$ for any $f\in F_N\cap B_{\|\cdot\|_{W^{1,1}}}(f_0,\beta_{K,N})$.
    \item $\forall f,g\in \mathcal D_1\cap F_N\cap B_{\|\cdot\|_{W^{1,1}}}(f_0,\beta_{K,N}),\, \|f-g\|_{W^{1,1}}\leq \beta_{K,N} \Rightarrow  \|D_f\tilde\L_{\eps,N}-D_g\tilde\L_{\eps,N}\|_{W^{1,1}}\leq \frac{1}{4\Lambda_{K,N}}$.
\end{enumerate}
Condition 1) and 2) are both satisfied as a result of the respective definitions of $\Lambda_{K,N}$ and $\delta_{K,N}$ and the fact that $\mathcal D_1\cap F_N\cap B_{\|\cdot\|_{W^{1,1}}}(f_0,\beta_{K,N})\subset B_{K}$.

Thus in the notation of \citep[Theorem]{bartle55}, setting the map $f$ to be $(\Pi_N-\tilde\L_{\eps,N})$ and the constant $\beta$ to be $\beta_{K,N}$, the convergence holds for any
$h^0\in B_{K,N}'$.
To guarantee that the set $B_{K,N}'$ contains $h_{\eps,N}^*$, 
 one can replace the previously fixed $K$ and $\beta$ with some $\tilde K\geq K+\delta_{K,N}C_N^{(2)}$ and $\beta_{\tilde K,N}=\min(1,\delta_{\tilde K,N})$. Since one can choose $\Lambda_{K,N}\leq \Lambda_{\tilde K,N}$, one has $\delta_{\tilde K,N}\leq\delta_{K,N}$, and thus $\beta_{\tilde K,N}\leq \beta_{K,N}$.  We show that $B_{\tilde K,N}'$ contains $h_{\eps,N}^*$.
Let $f\in \mathcal D_1\cap F_N\cap B_{\|\|_{W^{1,1}}}(h^*_{\eps,N},\beta_{\tilde K,N})$;  then $\| f\|_{W^{2,1}}\leq \|h^*_{\eps,N}\|_{W^{1,1}}+\beta_{\tilde K,N}C_N^{(2)}\le K+\beta_{\tilde K,N}C_N^{(2)}\leq \tilde K$, and therefore $f\in B_{\tilde K}$ and $ \mathcal D_1\cap F_N\cap B_{\|\|_{W^{1,1}}}(h^*_{\eps,N},\beta_{\tilde K,N})\subset B_{\tilde K}$.
As for the second condition, notice that the set $B_K\cap F_N\cap \{h: \|(\Pi_N-\tilde\L_{\eps,N})h\|_{W^{1,1}}< \frac{\beta_{\tilde K, N}}{2\Lambda_{\tilde K,N}} \}$ contains $h_{\eps,N}^*$ since $(\Pi_N-\tilde\L_{\eps,N})h_{\eps,N}^*=0$.} 
Thus $h_{\eps,N}^* \in B_{K,N}'$ .

Finally, according to the main theorem in \cite{bartle55}, the Newton scheme 
converges to $h_{\eps,N}^*$ as soon as $h^0\in B_{\tilde K,N}'\cap B_{\|\|_{W^{1,1}}}(h_{\eps,N}^*,\beta_{K,N})$.  The main theorem in \cite{bartle55} states convergence with exponential (order-1) speed. We proved above (see Equation \eqref{eq:speedconv}) that this convergence is ultimately order-$2$ convergence.
\end{proof}
\begin{remark}
\begin{itemize}
\,
    \item  For $\tilde K$ as in the statement of the Theorem above, the fixed point $h_{\eps,N}^*\in B_{\tilde K,N}'$ : Indeed, $\|(\Pi_N-\tilde\L_{\eps,N})h_{\eps,N}^*\|_{W^{1,1}}=0\leq \frac{\beta_{\tilde K, N}}{4\Lambda_{\tilde K,N}}$ and since $\|h_{\eps,N}^*\|_{W^{2,1}}\leq K < \tilde K -\beta_{\tilde K,N}C_N^{(2)}$ then $\mathcal D_1\cap F_N\cap B_{\|\|_{W^{1,1}}}(h_{\eps,N}^*,\beta_{\tilde K,N})\subset B_{\tilde K}$  and thus $h_{\eps,N}^*\in B_{\tilde K,N}'$.
    \item We did not assume uniqueness of a fixed point of $\tilde\L_{\eps,N}$. According to \citep[Theorem]{bartle55} the Newton scheme converges on $B_{K,N}'$ to some fixed point of $\tilde \L_{\eps,N}$, denoted $h_{\eps,N}^*$. If additionally, one initialises in $B_{K,N}'\cap B_{\|\|_{W^{1,1}}}(h_{\eps,N}^*,\beta_{K,N})$ by \citep[Theorem]{bartle55} the Newton scheme is guaranteed to converge to this particular $h_{\eps,N}^*$. 
    \item When $\eps\le \eps^*_4$, $h_{\eps,N}^*$ is unique (see Theorem \ref{thm:seqiter}), then 
 according to the Theorem in \cite{bartle55}, $B_{ K,N}'\subset B_{\|\cdot\|_{W^{1,1}}}(h_{\eps,N}^*,\beta_{ K,N})\cap \mathcal D_1\cap F_N$, which is consistent with previous point.
 \item According to Lemma \ref{lemcontinuityderiv}, $\delta_{K,N}\geq \frac{1}{4\Lambda_{K,N}\mathcal C_N}$ where $\mathcal C_N$ is the constant introduced in Lemma \ref{lemcontinuityderiv}. This leads to the following {lower bound} for the choice of $\beta_{K,N}$: $\frac{1}{4\Lambda_{K,N}\mathcal C_N}\leq \beta_{K,N}\leq \min\{1,\delta_{K,N}\}$.  
\end{itemize}

\end{remark}

In the following sections, we prove the existence of the Fr\'echet derivative of $\tilde \L_{\eps,N}$ and verify Conditions a) and b) with $\alpha=1$ for sufficiently small $\eps>0$ and sufficiently large $N$.

\subsection{Fr\'echet differentiability}\label{subsecdiff}
Recall the definition of the coupling map $I_f$ introduced in \eqref{eq:couplingmap}.
We introduce a few short technical lemmas before proving Fr\'echet derivative.
Using Lemma \ref{lem:product}, one can deduce the two following lemmas.
\begin{lemma}\label{lem:boudifprimeinv}
For $0\le \eps< \frac{1}{\|g\|_{C^1}}$, $f\in \mathcal D_1$, the map $x\in \mathbb{S}^1\mapsto (I_f')^{-1}(x):=\frac 1 {1+\eps\int\partial_1g(x,y)f(y)dy}$ belongs to $W^{4,1}$ and for any $k\in \{0,\dots,4\}$ (recall we assume $W^{0,1}:=L^1$), $I_f'\in W^{4,1}$ and there is a constant $R_k>0$, independent of $\eps$ and $f$, such that
\begin{align*}
    \left\| \frac{1}{I_f'} \right\|_{W^{k,1}}\leq R_k \frac{1}{(1-\eps\|g\|_{C^1})^{k+1}},\ \mbox{for $0\le k\le 4$ and all $f\in \mathcal D_1$.}
\end{align*}
\end{lemma}

\begin{proof}
We proceed by induction. Recall that $g\in C^k$ and by construction of $I_f$, the following estimates hold for any $k\in \{1,\dots,4\}$ : 
$$
1+\eps\|g\|_{C^k}\geq \|I_f^{(k)}\|_\infty \,\text{ and } |I_f'(x)|\geq 1-\eps \|g\|_{C^1} \,, \forall x\in \mathbb S^1.
$$ 
Consequently, we can initialize the induction at $k=0$ with $\|\frac{1}{I_f'}\|_{L^1}\leq \|\frac{1}{I_f'}\|_{\infty}\leq  \frac{1}{1-\eps \|g\|_{C^1}}$. We obtain 
\begin{align*}
    \left\|\frac{1}{I_f'}\right\|_{L^1}\leq \left\|\frac{1}{I_f'}\right\|_{\infty} \leq \frac{R_0}{1-\eps \|g\|_{C^1}},\\  
\end{align*}
and so we may take $R_0=1$.
To proceed by induction we assume that there is $n\leq 3$ such that for any $k\le n$,  
$$\left\| \frac{1}{I_f'} \right\|_{W^{k,1}}\leq R_k \frac{1}{(1-\eps\|g\|_{C^1})^{k+1}}.$$ 
Then, using Lemma \ref{lem:product}, there is $K_{n+1}>0$ such that
\begin{align*}
    \|\left(\frac{1}{I_f'}\right)^{(n+1)}\|_{L^1}&= \left\|\left(\left(\frac{1}{I_f'}\right)'\right)^{(n)}\right\|_{L^1}\\
    &= \left\|\left(\frac{I_f''}{(I_f')^2}\right)^{(n)}\right\|_{L^1}\\
    &\leq \sum_{k=0}^{n} {n \choose k}\left\|\left(\frac{I_f''}{I_f'}\right)^{(n-k)}\left(\frac{1}{(I_f')}\right)^{(k)} \right\|_{L^1}\\
        &\leq \sum_{k=1}^{n}{n \choose k}\left\|\left(\frac{I_f''}{I_f'}\right)^{(n-k)}\right\|_{W^{1,1}}\left\|\left(\frac{1}{(I_f')}\right)^{(k)} \right\|_{L^1}\\
   & +\left\|\left(\frac{I_f''}{I_f'}\right)^{(n)}\right\|_{L^1}\left\|\frac{1}{(I_f')} \right\|_{W^{1,1}}\\
   &\leq \sum_{k=1}^{n}2^{n-k+1}{n \choose k}\left\|I_f''\right\|_{W^{n-k+1,1}}\left\|\frac{1}{I_f'}\right\|_{W^{n-k+1,1}}\left\|\frac{1}{(I_f')} \right\|_{W^{k,1}}\\
   & +2^n\left\|I_f''\right\|_{W^{n,1}}\left\|\frac{1}{I_f'}\right\|_{W^{n,1}}\left\|\frac{1}{(I_f')} \right\|_{W^{1,1}}\\
    &\leq \sum_{k=1}^{n}2^{n-k+1}{n \choose k}\|I_f''\|_{W^{n-k+1,1}}\frac{R_{n-k+1}}{(1-\eps \|g\|_{C^1})^{n-k+1}}\frac{R_{k}}{(1-\eps \|g\|_{C^1})^{k+1}}\\
     & +2^n\|I_f''\|_{W^{n,1}}\frac{R_{n}}{(1-\eps \|g\|_{C^1})^{n+1}}\frac{R_{1}}{(1-\eps \|g\|_{C^1})}\\
    &\leq \frac{K_{n+1}}{(1-\eps \|g\|_{C^1})^{n+2}},
\end{align*}
Where we used the fact that $I_f''=\eps\int g''(x,y)f(y)dy$ and thus $\|I_f''\|_{W^{n,1}} \leq \eps \|g\|_{C^{n+2}}$.
Consequently, using the induction property along with the above equation, there is $R_{n+1}>0$ such that 
\begin{align*}
    \left\|\frac{1}{I_f'}\right\|_{W^{n+1,1}}\leq \left\|\left(\frac{1}{I_f'}\right)^{(n+1)}\right\|_{L^1}+ \left\|\frac{1}{I_f'}\right\|_{W^{n,1}}\leq \frac{R_{n+1}}{(1-\eps \|g\|_{C^1})^{n+2}},
\end{align*}
which completes the proof of the lemma.
\end{proof}
\begin{lemma}\label{lem:LIfcon}
    For $0\le\eps< \frac{1}{\|g\|_{C^1}}$, $f\in \mathcal D_1$, the linear map $\L_{I_f}\in L(W^{1,1},W^{1,1})$ is continuous. Moreover, there is $C_0>0$, independent of $\eps$ and $f\in\mathcal D_1$ such that for $h\in W^{1,1}$,
\begin{align*}
    \|\L_{I_f}h\|_{W^{1,1}}\leq C_0\|h\|_{W^{1,1}}.
\end{align*}

\end{lemma}

\begin{proof}

    Indeed notice that for $h\in W^{1,1}$,
    
    \begin{align}
    \left(\L_{I_f}(h)\right)'&=\left(\frac{h}{I_f'}\circ I_f^{-1}\right)'\nonumber\\
    &=\left(\frac{h}{I_f'}\right)'\circ I_f^{-1}\cdot \frac{1}{I_f'}\circ I_f^{-1}\nonumber\\
    &=\L_{I_f}\left(\left(\frac{h}{I_f'}\right)'\right).\label{eq:Lifh}
\end{align}

  Now recall that for $h\in W^{1,1}$, $\|\L_{I_f}h\|_{L^1}\leq\|h\|_{L^1}$. Notice that  using \eqref{eq:Lifh}, Lemma \ref{lem:boudifprimeinv} and Lemma \ref{lem:product}
\begin{align*}
    \|\left(\L_{I_f}h\right)'\|_{L^1}&= \left\|\L_{I_f}\left(\left(\frac{h}{I_f'}\right)'\right)\right\|_{L^1}\\
    &=\left\|\left(\frac{h}{I_f'}\right)'\right\|_{L^1}\\
        &\leq 2\|h\|_{W^{1,1}}\|\frac{1}{I_f'}\|_{W^{1,1}}\\
         &\leq 2\|h\|_{W^{1,1}}\frac{R_1}{(1-\eps \|g\|_{C^1})^2}.
\end{align*}
Thus there is $C_0>0$ such that 
\begin{align*}
    \|\L_{I_f}h\|_{W^{1,1}}\leq C_0\|h\|_{W^{1,1}}.
\end{align*}
\end{proof}
We now prove the main proposition of this section. In the rest of section \ref{sec:newtheory} we introduce the notation $G: h\in W^{k,1}\mapsto C^5$ given by the convolution with the kernel $g\in C^5$ introduced in Section \ref{sec:mainresult},
\begin{align*}
    G(h):=\int_{\mathbb S^1} g(x,y)h(y)dy.
\end{align*}

\begin{proposition}\label{prop:deriv}
For $\eps< \frac{1}{\|g\|_{C^1}}$, 
$f\in B_K\cap F_N$ and $e\in F_N$, the Fr\'echet derivative  $D_f\tilde{\mathcal L}_{\eps,N}\in L(F_N,F_N)$  has the representation 
\begin{align}
\label{fullderiveqn}
    D_f\tilde\L_{\eps,N}(e)=\Pi_N\L_{T_{\eps, f}}e + \Pi_N \left[\mathcal J(f)\right](e),
\end{align}
where for $h\in  F_N$, 
\begin{align}\label{eqdeltafif}
[\mathcal J(h)](e):=\partial_f(\L_{T_{\eps, f}}h)(e)=
    \partial_f(\L_{T_{\eps, f}}h)(e)=-\eps\L_{I_f}\left[\left(\frac{\L_{T}(h)\cdot G(e)}{|I_{f}'|}\right)'\right].    
\end{align}

Furthermore, the following Taylor expansion holds: for any $e\in F_N$,
\begin{align}\label{eq:taylordisc}
    \|\tilde\L_{\eps,N} (f+e) -\tilde \L_{\eps,N}(f) -D_f \tilde \L_{\eps,N} (e)\|_{W^{1,1}}=(1+\|f\|_{W^{3,1}})o(\|e\|_{W^{1,1}}).
\end{align}

\end{proposition}

\begin{proof}
We prove the existence of the Fr\'echet derivative through the development of \eqref{eq:taylordisc}. 
Recall $f\in B_K\cap F_N$ and $e\in F_N$.
The Frechet derivative of $\Pi_N\tilde{\L}_\eps$ corresponds to the first-order term in $e$ in the development of $\Pi_N\tilde{\L}_\eps(f+e)$ and upper bounding the other terms will give us the Taylor expansion \eqref{eq:taylordisc}.
To this end, note
\begin{equation}
\Pi_N\tilde{\L}_\eps(f+e)=\Pi_N\tilde{\L}_\eps f+\Pi_N(\L_{T_{\eps, f+e}}-\L_{T_{\eps, f}})f+\Pi_N\L_{T_{\eps, f}}e+\Pi_N(\L_{T_{\eps, f+e}}-\L_{T_{\eps, f}})e.
\label{eq:lepsderiv}
\end{equation}
Thus to prove the Taylor expansion \eqref{eq:taylordisc}, we show below it is enough to prove the derivative $(\partial_f\L_{T_{\eps, f}}h)$ on $W^{1,1}$ is well defined and the asymptotic \eqref{eqderivo} holds for any $h,f\in B_K\cap F_N$ and $e\in F_N$ 
\begin{align}\label{eqderivo}
    \|(\L_{T_{\eps, f+e}}-\L_{T_{\eps, f}})h-(\partial_f\L_{T_{\eps, f}}h)(e)\|_{W^{1,1}}=o(\|e\|_{W^{1,1}}).
\end{align}
Assuming \eqref{eqderivo}, one has, using the desired expression for $D_f\tilde{\L}_{\eps,N}$ as in \eqref{fullderiveqn}, and using \eqref{eq:lepsderiv},
\begin{align}
&\|\Pi_N\tilde{\L}_\eps(f+e)-\Pi_N\tilde{\L}_\eps f-\Pi_N\L_{T_{\eps, f}}e-\Pi_N\left[\mathcal J(f)\right](e)\|_{W^{1,1}}\nonumber\\
\leq& \|\Pi_N(\L_{T_{\eps, f+e}}-\L_{T_{\eps, f}})f-\Pi_N\left[\mathcal J(f)\right](e)\|_{W^{1,1}}+\|\Pi_N(\L_{T_{\eps, f+e}}-\L_{T_{\eps, f}})e\|_{W^{1,1}}\nonumber\\
\leq& \|\Pi_N(\L_{T_{\eps, f+e}}-\L_{T_{\eps, f}})f-\Pi_N\left[\mathcal J(f)\right](e)\|_{W^{1,1}}+\|\Pi_N(\L_{T_{\eps, f+e}}-\L_{T_{\eps, f}})e-\Pi_N\left[\mathcal J(e)\right](e)\|_{W^{1,1}}\nonumber\\
&+\|\Pi_N\left[\partial_f\L_{T_{\eps, f}}e\right](e)\|_{W^{1,1}}\nonumber\\
\leq & o(\|e\|_{W^{1,1}})+\|\left[\partial_f\L_{T_{\eps, f}}e\right](e)\|_{W^{1,1}}\label{eq:devtildelderiv}\\
&=o(\|e\|_{W^{1,1}}).\label{eq:devtildelderiv2}
\end{align}
It remains to show that $\|(\partial_f\L_{T_{\eps, f}}e)(e)\|_{W^{1,1}}=o(\|e\|_{W^{1,1}}).$
We will show this later 
(see equation \eqref{eqderivif} replacing $h$ by $\L_Te$ and the computation in \eqref{eq:boundldeuxerror}).

We will prove equation \eqref{eqderivo} in two steps, first we prove that 
\begin{align*}
    \|(\L_{T_{\eps, f+e}}-\L_{T_{\eps, f}})h-(\partial_f\L_{T_{\eps, f}}h)(e)\|_{L^{1}}=o(\|e\|_{L^1}),
\end{align*}
 and in a second step we use it to prove
\begin{align*}
    \|(\L_{T_{\eps, f+e}}-\L_{T_{\eps, f}})h-(\partial_f\L_{T_{\eps, f}}h)(e)\|_{W^{1,1}}=o(\|e\|_{L^1}).
\end{align*}

\textbf{Step 1: Proof of Formula \eqref{eqderivo} with $L^1$ norm}

Since $\L_{T_{\eps, f}}=\L_{I_f}\L_{T}$,
\begin{align*}
    \partial_f \L_{T_{\eps, f}}h=(\partial_f \L_{I_f})\L_{T}h.
\end{align*}
We consider the derivative of 
$$\L_{I_f}: h \mapsto \left(\frac{h}{|I_f'|}\right)\circ I_f^{-1},$$
where
$$I_f: x\in \mathbb S^1 \mapsto x+ \eps G(f).$$ 
Notice that if we assume $\eps<\frac{1}{\|g\|_{C^1}}$ 
then, 
for $f\in B_K$, one has $\|f\|_{L^1}=1$  and $\eps \int \partial_1g(x,y)f(y)dy\leq \eps \|g\|_{C^1}\leq 1$ for any $x\in \mathbb S ^1$. Thus $|I_f'|=I_f'$.
Let $F(f,x):=(\frac{h}{|I_f'|})(x)$ for $f\in B_K$ and $x\in \mathbb S^1$, then
\begin{equation*}
    \L_{I_f}h(x)=(\frac{h}{|I_f'|})\circ I_f^{-1}(x):=F(f,I_f^{-1}(x)).
\end{equation*}
Thus for $x\in \mathbb S^1$,
\begin{align}
\label{fullderivF}
    \partial_f(F(f,I_f^{-1}(x)))=\partial_2 F(f,I_f^{-1}(x))\partial_f(I_f^{-1}(x))+\partial_1F(f,I_f^{-1}(x)).
\end{align}
We start with the computation of $\partial_1F(f,x)$ 
\begin{align*}
    F(f+e,.)&=\frac{h}{|I_{f+e}'|}\\
    &=\frac{h}{|I_{f}'|}\left(\frac{1}{1+(\eps G(e))'/|I_{f}'|}\right)\\
    &=\frac{h}{|I_{f}'|}\left(1-\frac{(\eps G(e))'}{|I_{f}'|}\right)+\frac{h}{|I_{f}'|}\sum_{k\geq 2} \left(-\frac{(\eps G(e))'}{|I_{f}'|}\right)^k\\
   &=\frac{h}{|I_{f}'|}-\frac{h}{|I_{f}'|^2}(\eps G(e))'+\frac{h}{|I_{f}'|}O(\|e\|_{L^1}^2),
\end{align*}
where we used in the lines above from the expression of $Ge$ the following fact that 
\begin{align}\label{eq:boundG}
\|G(e)\|_{C^k}&\leq \|\int g(x,y)e(y)dy\|_{C^{k}}\nonumber\\
&\leq \|g\|_{C^k}\|e\|_{L^1}.
\end{align}

Thus 
\begin{align}\label{eqdunf}
    \partial_1F(f,x)[e]=-\frac{h}{|I_{f}'|^2}(\eps G(e))'(x)
\end{align}
and
\begin{align}\label{eq:partialf}
\|F(f+e,\cdot)-F(f,\cdot)-\partial_1F(f,\cdot)[e]\|_\infty=\|h\|_{L^1}O(\|e\|_{L^1}^2).
\end{align}

As for $\partial_2F(f,x)$, it is a regular derivative on $(\mathbb R,|\cdot|)$. Recall that we have chosen $\eps$ sufficiently small so that  for $f\in B_K$ we have $|I_f'|=I_f'$, thus:
\begin{align}\label{eqddeuxf}
    \partial_2F(f,x)=\frac{h'I_f'-hI_f''}{I_f'^2},
\end{align}
and since $\mathbb S^1$ is compact, $h\in W^{3,1}$, and $x\mapsto I_f(x)\in C^5$, one has $x\mapsto F(f,x)$ is $C^1$ and 
\begin{align*}
\|F(f,x+\delta)-F(f,x)-\partial_2F(f,x)\delta\|_{L^1}=\|h\|_{W^{2,1}}O(\delta^2).    
\end{align*}

We are left to compute $\partial_f (I_f^{-1})(x)$. To do so, consider $I_f^{-1}$ as a two-coordinate function $(f,x)\mapsto I_f^{-1}(x)$. Under this formalism,
\begin{align*}
    0=\partial_f(I_f^{-1}\circ I_f(x))=\partial_1I_f^{-1}(I_f(x))+\partial_2I_f^{-1}(I_f(x))[\partial_fI_f(x)].
\end{align*}
Thus,
\begin{align*}
\partial_1I_f^{-1}(I_f(x))=-\partial_2I_f^{-1}(I_f(x))[\partial_fI_f(x)]   
\end{align*}
and since $\partial_2I_f^{-1}(I_f(x))=(\partial_2I_f)^{-1}$ in $(\mathbb R,|\cdot|)$,
\begin{align*}
\left[\partial_1I_f^{-1}(I_f(x))\right](e)&=\left[-(\partial_2I_f(x))^{-1}\partial_fI_f(x)\right](e)\\
&=-\frac{\eps G e}{I_f'}(x).
\end{align*}
Thus, 
\begin{align*}
    \left[\partial_1(I_f^{-1}(x))\right](e)&=\left[-\left[(\partial_2I_f(\cdot))^{-1}\partial_fI_f(\cdot)\right]\circ I_f^{-1}(x)\right](e)\\
&=\left[-\frac{\eps G e}{I_f'}\circ  I_f^{-1}(x)\right](e).
\end{align*}
In other words, for any fixed $x\in \mathbb S^1$
\begin{align}\label{eq:derivifsansexp}
I_{f+e}^{-1}(I_f(x))=I_f^{-1}(I_f(x))-\frac{\eps G e}{I_f'}(x)+o_x(\|e\|_{L^1}).
\end{align}
Notice that since $\|I_f'\|_\infty\geq 1-\eps\|g\|_{C^1}$, the derivative is itself bounded as follows,
\begin{align}\label{eq:deltaboundf}
    \left\|\frac{\eps G e}{I_f'}\right\|_{\infty}\leq \frac{\|g\|_{C^1}}{1-\eps\|g\|_{C^1}}\|e\|_{L^1}
\end{align}
For $e\in F_N$,
for each $x\in \mathbb S^1$ one can apply a standard mean value theorem to state that there is $e'_x:=\delta_x e$ with $\delta_x\in (0,1)$ such that
\begin{align*}
    \frac{I_f^{-1}(x)-I_{f+e}^{-1}(x)}{\|e\|_{L^1}}=\partial_1(I_{f+e'_x}^{-1}(x))[\frac{e}{\|e\|_{L^1}}].
\end{align*}
Thus for any $x\in \mathbb S^1$, by the uniform continuity of $f\mapsto \partial_1(I_f^{-1}(x))[e]$,
\begin{align}
    \left|\frac{I_f^{-1}(x)-I_{f+e}^{-1}(x)}{\|e\|_{L^1}}-\partial_1(I_f^{-1}(x))\left[\frac{e}{\|e\|_{L^1}}\right]\right|&=\left|\partial_1(I_{f+e'_x}^{-1}(x))\left[\frac{e}{\|e\|_{L^1}}\right]-\partial_1(I_f^{-1}(x))\left[\frac{e}{\|e\|_{L^1}}\right]\right|\nonumber\\
    &=\left|\frac{\eps G (\frac{e}{\|e\|_{L^1}})}{I_{f+e'_x}'}\circ  I_{f+e'_x}^{-1}(x)-\frac{\eps G (\frac{e}{\|e\|_{L^1}})}{I_f'}\circ  I_f^{-1}(x)\right|,\label{eq:littlederiv}
\end{align}
Notice that for any $x\in \mathbb S^1$, $\|I_f-I_{f+e}\|_\infty\leq \eps\|g\|_\infty \|e\|_{L^1}$
and since $x\in \mathbb S^1\mapsto I_f(x)$ is $C^1$ and invertible, we deduce that $\lim_{\|e\|_{L^1}\to 0}\|I_f^{-1}-I_{f+e}^{-1}\|_\infty=0$.
\begin{align*}
    \lim_{\|e\|_{L^1}\to 0}\|I_f^{-1}-I_{f+e}^{-1}\|_{\infty}&=\lim_{\|e\|_{L^1}\to 0}\|I_{f}^{-1}\circ I_{f+e}\circ I_{f+e}^{-1}-I_{f}^{-1}\circ I_{f}\circ I_{f+e}^{-1}\|_\infty\\
    &=\lim_{\|e\|_{L^1}\to 0}\|I_{f}^{-1}\circ I_{f+e}-I_{f}^{-1}\circ I_{f}\|_\infty=0.
\end{align*}
Thus combining this with equation \eqref{eq:littlederiv} we obtain the required equation \eqref{eq:derivifsansexp} uniformly in $x$
\begin{align*}
    \|I_f^{-1}(x)-I_{f+e}^{-1}(x)-\partial_1(I_f^{-1}(x))[e]\|_\infty=o(\|e\|_{L^1}).
\end{align*}

Notice from the expression \eqref{eqddeuxf} of $\partial_2F(f,\cdot)$ that  $f\in( B_K\cap F_N,\|\cdot\|_{L^1})\mapsto \partial_2F(f,\cdot) \in L(\mathbb R, L^1)$ is continuous.
The final result of \textbf{Step 1}, that is equation \eqref{eqderivo} with the $L^1$ norm replacing the $W^{1,1}$ norm,
is then obtained by recombining everything 
using first equation \eqref{eqddeuxf} and the continuity of $\partial_2F(f,\cdot)$, then \eqref{eq:derivifsansexp} with the boundedness of the derivative $\partial_fI_f^{-1}$ (see equation \eqref{eq:deltaboundf}) and finally the expansion \eqref{eq:partialf}
in the computation that follows. For brevity, we will write $F(f,I_f^{-1})$ instead of $F(f,\cdot)\circ I_f^{-1}$ in the equations below.
\begin{align*}
    F(f+e,I_{f+e}^{-1})&=F(f+e,I_f^{-1})+ [\partial_2F(f+e,I_f^{-1})](I_{f+e}^{-1}-I_f^{-1}) +\|h\|_{W^{2,1}}O_{L^1}(\|I_{f+e}^{-1}-I_f^{-1}\|_{\infty}^{2})\\
    &=F(f+e,I_f^{-1})+ [\partial_2F(f+e,I_f^{-1})](\partial_fI_f^{-1}(e))+o_{L^1}(\|e\|_{L^1}) \\
    &+\|h\|_{W^{2,1}}O_{L^1}(\|\partial_fI_f^{-1}(e)+o_{\infty}(\|e\|_{L^1})\|_{\infty}^{2})\\
    &=F(f,I_f^{-1})+\partial_1F(f,I_f^{-1})+ [\partial_2F(f+e,I_f^{-1})](\partial_fI_f^{-1}(e))+(1 +\|h\|_{W^{2,1}})o_{L^1}(\|e\|_{L^1}))\\
     &=F(f,I_f^{-1})+\partial_1F(f,I_f^{-1})+ [\partial_2F(f,I_f^{-1})](\partial_fI_f^{-1}(e))+(1 +\|h\|_{W^{2,1}})o_{L^1}(\|e\|_{L^1})).
\end{align*}
The latter display can be summarized as
\begin{align}\label{eq:derivlif}
    \|\L_{I_{f+e}}h-\L_{I_{f}}h-\partial_f(\L_{I_{f}}h)[e]\|_{L^1}=(1+\|h\|_{W^{2,1}})o(\|e\|_{L^1}).
\end{align}
Using equation \eqref{fullderivF} along with the formulae \eqref{eqdunf}, \eqref{eqddeuxf} and \eqref{eq:derivifsansexp}  we may express $\partial_f(\L_{I_{f}}h)$ as
\begin{align}
   \left[\partial_f(\L_{I_{f}}h)(x)\right](e)&=\left[\partial_2 F(f,I_f^{-1}(x))\partial_f(I_f^{-1}(x))\right](e)+\left[\partial_1F(f,I_f^{-1}(x))\right](e)\nonumber\\
     &=\left[-\left(\frac{h'I_f'-hI_f''}{|I_f'|^2}\right) \frac{\eps G e}{I_f'}-\frac{h}{|I_{f}'|^2}(\eps Ge)'\right]\circ I_f^{-1}(x))\nonumber\\
 &=-\left[\frac{1}{I_f'}\left[\left(\frac{h}{|I_{f}'|}\right)' \eps G e+\frac{h}{|I_{f}'|}(\eps Ge)'\right]\right]\circ I_f^{-1}(x))\nonumber\\
&=-\left[\frac{1}{I_f'}\left(\frac{h}{|I_{f}'|} \eps G e\right)'\right]\circ I_f^{-1}(x))\nonumber\\
&=-\L_{I_f}\left[\left(\frac{h}{|I_{f}'|} \eps G e\right)'\right](x).\label{eqderivif}
\end{align}
Thus,
\begin{align*}
    \|(\L_{T_{\eps, f+e}}-\L_{T_{\eps, f}})h-(\partial_f\L_{T_{\eps, f}})h\|_{L^1}=(1+\|\mathcal{L}_Th\|_{W^{2,1}})o(\|e\|_{L^{1}(m)})=(1+\|h\|_{W^{2,1}})o(\|e\|_{L^{1}(m)}),
\end{align*}
with 
\begin{equation}
    \label{Tepsderiv}
\partial_f(\L_{T_{\eps, f}}h)(e)=\partial_f(\L_{I_{f}}\circ\L_Th)(e)=-\L_{I_f}\left[\left(\frac{\L_{T}h}{|I_{f}'|} \eps G e\right)'\right].
\end{equation}

\textbf{Step 2: Proof of Formula \eqref{eqderivo} with $W^{1,1}$ norm}

We now prove additional regularity, that is
\begin{align*}
    \|(\L_{T_{\eps, f+e}}-\L_{T_{\eps, f}})h-(\partial_f\L_{T_{\eps, f}})h\|_{W^{1,1}}=o(\|e\|_{W^{1,1}}).
\end{align*}
Recall that by \eqref{eq:Lifh}, for $h\in W^{2,1}$,
\begin{align*}
    \left(\L_{I_f}(h)\right)'=\L_{I_f}(\left(\frac{h}{I_f'}\right)').
\end{align*}
Thus,
\begin{align}
    \left(\left(\L_{I_{f+e}}-\L_{I_f}\right)h\right)'&=
\L_{I_{f+e}}\left(\left(\frac{h}{I_{f+e}'}\right)'\right)-\L_{I_f}\left(\left(\frac{h}{I_f'}\right)'\right)\nonumber\\
    &=(\L_{I_{f+e}}-\L_{I_f})\left(\left(\frac{h}{I_f'}\right)'\right)+\L_{I_{f+e}}\left(\left(\frac{h}{I_{f+e}'}\right)'-\left(\frac{h}{I_f'}\right)'\right).\label{eqderivnorm}
\end{align}
Also note that, using \eqref{eq:derivlif} and substituting $(h/I_f')'$ for $h$ in \eqref{eqderivif}, we obtain, using the bound \eqref{eq:boundG} to move the derivatives of $e$ onto $g$ in the definition of $G$,
\begin{align}\label{eqtermunderiv}
    \left\|(\L_{I_{f+e}}-\L_{I_f})\left(\left(\frac{h}{I_f'}\right)'\right)-\L_{I_f}\left(\left(\frac{1}{I_f'}\left(\frac{h}{I_f'}\right)'\eps Ge\right)'\right)\right\|_{L^1}=(1+\|h\|_{W^{3,1}})o(\|e\|_{L^1}).
\end{align}
As for the second term in \eqref{eqderivnorm}, let $H:=\left(\frac{h}{I_{f+e}'}\right)'-\left(\frac{h}{I_f'}\right)'$, then according to formula \eqref{eq:derivlif}
\begin{align}\label{eq:Hderiv}
&\left\|\L_{I_{f+e}}\left(H\right)-\L_{I_f}\left(H\right)-\partial_f(\L_{I_f}H)[e]\right\|_{L^1}=(1+\|H\|_{W^{2,1}})o(\|e\|_{L^1})=(1+\|h\|_{W^{3,1}})o(\|e\|_{L^1})
\end{align}
To estimate $\L_{I_{f+e}}\left(H\right)$ it is thus enough to estimate $\partial_f(\L_{I_f}H)[e]$ and $\L_{I_f}\left(H\right)$.
We start with an estimate of $H$ in $\|\cdot\|_\infty$ where we used the bound \eqref{eq:boundG}: 
\begin{align}
  \|H\|_\infty&=\left\| \left(\frac{h}{I_{f+e}'}\right)'-\left(\frac{h}{I_f'}\right)' \right\|_\infty=\left\|\left[h\left(\frac{1}{I_{f}'+\eps (Ge)'}-\frac{1}{I_{f}'}\right)\right]'\right\|_\infty\nonumber\\
  &=\left\|\left[\frac{h}{I_{f}'}\sum_{k\geq 1}\left(-\frac{\eps (Ge)'}{I_{f}'}\right)^k\right]'\right\|_\infty=O(\|e\|_{L^1}).\label{Hinfbound}
\end{align}

Using \eqref{Hinfbound}, we now estimate $\partial_f(\L_{I_f}H)[e]$, according to equation \eqref{eqderivif} :
\begin{align}\label{eq:grandhderiv}
    \|\partial_f(\L_{I_f}H)[e]\|_\infty&=\left\|\L_{I_f}\left[\left(\frac{H}{I_{f}'} \eps G e\right)'\right] \right\|_\infty\nonumber\\
    &=\left\|\L_{I_f}\left[\left(\left[\frac{h}{I_{f}'}\sum_{k\geq 1}\left(-\frac{\eps (Ge)'}{I_{f}'}\right)^k\right]'\frac{\eps G e}{I_{f}'} \right)'\right]\right\|_\infty\nonumber\\
    &=\left\|\frac{1}{I_{f}'}\left[\left(\left[\frac{h}{I_{f}'}\sum_{k\geq 1}\left(-\frac{\eps (Ge)'}{I_{f}'}\right)^k\right]'\frac{\eps G e}{I_{f}'} \right)'\right]\right\|_\infty\nonumber\\
    &\leq C\|h\|_{C^2}\left\|\left(\left[\sum_{k\geq 1}\left(-\frac{\eps (Ge)'}{I_{f}'}\right)^k\right]'\frac{\eps G e}{I_{f}'} \right)'\right\|_\infty\nonumber\\
    &=o(\|e\|_{L^1}).
\end{align}
We now compute $\L_{I_f}\left(H\right)$: using \eqref{Hinfbound}, one has
\begin{align}\label{eqtermdeuxderiv}
\L_{I_f}\left(H\right)&= \L_{I_f}\left(\left[\frac{h}{I_{f}'}\sum_{k\geq 1}\left(-\frac{\eps (Ge)'}{I_{f}'}\right)^k\right]'\right)\nonumber\\
&=-\L_{I_f}\left(\left[\frac{h}{I_{f}'}\frac{\eps (Ge)'}{I_{f}'}\right]'\right)+o_{L^1}(\|e\|_{L^1})
\end{align}
Returning to the second term in equation \eqref{eqderivnorm}, using the definition of $H$,  \eqref{eq:Hderiv}, \eqref{eq:grandhderiv}, and \eqref{eqtermdeuxderiv},  we know that 
{\begin{align*}
    \L_{I_{f+e}}\left(\left(\frac{h}{I_{f+e}'}\right)'-\left(\frac{h}{I_f'}\right)'\right)&=\L_{I_{f+e}}(H)\\
    &=\L_{I_{f}}(H)+\partial_f(\L_{I_f}H)[e]+(1+\|H\|_{W^{2,1}})o_{L^1}(\|e\|_{L^1})\\
    &=\L_{I_{f}}(H)+(1+\|h\|_{W^{3,1}})o_{L^1}(\|e\|_{L^1})\\
 &=-\L_{I_f}\left(\left[\frac{h}{I_{f}'}\frac{\eps (Ge)'}{I_{f}'}\right]'\right)+(1+\|h\|_{W^{3,1}})o_{L^1}(\|e\|_{L^1}),
\end{align*}
The left hand side of the equation above matches with the right term in the right hand side of equation \eqref{eqderivnorm}, thus along with equation \eqref{eqtermunderiv}, the expression in equation \eqref{eqderivnorm} becomes,
}

\begin{align}\label{eqderivnorm2}
    \left(\left(\L_{I_{f+e}}-\L_{I_f}\right)h\right)'&= -\L_{I_f}\left(\left(\frac{1}{I_f'}\left(\frac{h}{I_f'}\right)'\eps Ge\right)'\right)-\L_{I_f}\left(\left[\frac{h}{I_{f}'}\frac{\eps (Ge)'}{I_{f}'}\right]'\right)+(1+\|h\|_{W^{3,1}})o_{L^1}(\|e\|_{W^{1,1}}).
\end{align}
Also note,
\begin{align}
 &-\L_{I_f}\left(\left(\frac{1}{I_f'}\left(\frac{h}{I_f'}\right)'\eps Ge\right)'\right)-\L_{I_f}\left(\left(\frac{1}{I_f'}\left(\frac{h}{I_{f}'}(\eps Ge)'\right)\right)'\right)\nonumber\\
 &=-\L_{I_f}\left(\left(\frac{1}{I_f'}\left[\left(\frac{h}{I_f'}\right)'\eps Ge+\frac{h}{I_{f}'}(\eps Ge)'\right]\right)'\right)\nonumber\\
 &=-\L_{I_f}\left[\left(\frac{1}{I_f'}\left(\frac{h}{I_{f}'} \eps G e\right)'\right)'\right]\nonumber\\
 &=\left(-\L_{I_f}\left[\left(\frac{h}{I_{f}'} \eps G e\right)'\right]\right)'\nonumber\\
 &= \left[(\partial_f\L_{T_{\eps, f}})h\right]', 
 \label{LTepsderiv2}
\end{align}
where the final equality follows from \eqref{Tepsderiv}.
Therefore, for any $h\in W^{3,1}$,
using \eqref{eqderivnorm2} and \eqref{LTepsderiv2} one obtains
\begin{align}\label{eq:derivif}
    \|(\L_{I_{ f+e}}-\L_{I_{ f}})h-(\partial_f\L_{I_{ f}}h)[e]\|_{W^{1,1}}=(1+\|h\|_{W^{3,1}})o(\|e\|_{W^{1,1}}).
\end{align}
 This concludes the proof of \eqref{eqderivo}.

We now turn to justifying 
equation \eqref{eq:devtildelderiv}--\eqref{eq:devtildelderiv2}. Using Lemmas \ref{lem:product}, \ref{lem:boudifprimeinv} and \ref{lem:LIfcon} in the formula \eqref{eqderivif} for the derivative $(\partial_f\L_{T_{\eps, f}}e)(e)$ we obtain the missing estimate for equation \eqref{eq:devtildelderiv}--\eqref{eq:devtildelderiv2}:

\begin{align}\label{eq:boundldeuxerror}
\|(\partial_f\L_{T_{\eps, f}}e)(e)\|_{W^{1,1}}&=\left\|\L_{I_f}\left[\left(\frac{\L_Te}{|I_{f}'|} \eps G e\right)'\right]\right\|_{W^{1,1}}\\
&\leq C_0\left\|\left[\left(\frac{\L_Te}{|I_{f}'|} \eps G e\right)'\right]\right\|_{W^{1,1}}\nonumber\\
&\leq 4C_0\left\|\frac{\L_Te}{|I_{f}'|}\right\|_{W^{2,1}} \left\|\eps G e\right\|_{W^{2,1}}\nonumber\\
&\leq 4C_0\left\|\frac{\L_Te}{|I_{f}'|}\right\|_{W^{2,1}} \eps\|g\|_{C^2}\|e\|_{L^1}\nonumber\\
&\leq 16C_0 \mathcal C\|e\|_{W^{2,1}}\left\| \frac{1}{I_f'} \right\|_{W^{2,1}} \eps\|g\|_{C^2}\|e\|_{L^1}\nonumber\\
&\leq 16\mathcal CR_2C_N'\|e\|_{L^1}\frac{1}{(1-\eps\|g\|_{C^1})^3} \eps\|g\|_{C^2}\|e\|_{L^1},\nonumber
\end{align}
where $C_N'$ is the constant corresponding to the equivalence of norm on $F_N$ so that $\|e\|_{W^{2,1}}\leq C_N'\|e\|_{L^1}$ 
for any $e\in F_N$.
This inserted in equation \eqref{eq:devtildelderiv} concludes the Proposition.

\end{proof}

\subsection{Lipschitz regularity of the derivative}

Lemma \ref{lemcontinuityderiv} proves the Lipschitz regularity of $f \in F_N\cap B_K\mapsto D_f\tilde \L_{\eps,N} \in L(F_N,F_N)$. 
\begin{lemma}
\label{lemcontinuityderiv}
For $\eps< \max\{\frac{1}{32R_2},\frac{1}{2^{1/3}\|g\|_{C^1}}\}$  and $N>0$, the map $f\in B_K\cap F_N\mapsto D_f\tilde \L_{\eps,N} \in L(W^{2,1},W^{1,1})$, is continuous and there is $\mathcal C_N>C_N^{(2)}$ (where $C_N^{(2)}>0$ holds for the equivalence constant between $W^{1,1}$ and $W^{2,1}$ see \eqref{eq:equivcon}) such that for 
any $f\in B_K \cap F_N$, any $\tilde f\in B_{\|\cdot\|_{W^{1,1}}}(0,1)\cap F_N$ and any $e\in F_N$,

\begin{align}
    \|D_{f}(\tilde \L_{\eps,N})(e) -D_{f+\tilde f}(\tilde \L_{\eps,N})(e)\|_{W^{1,1}}\le \mathcal C_N\|\tilde f\|_{W^{1,1}}\|e\|_{W^{1,1}}.\label{eq:holderderiv}
\end{align}

\end{lemma}

\begin{proof}
Let $f\in B_K\cap F_N$, and recall from Proposition \ref{prop:deriv}  the explicit formula for $D_f(\Pi_N \tilde \L_\eps)$ : for any $f\in B_K\cap F_N$ and $e\in F_N$,
\begin{align}\label{eq:descripderiv}
D_f(\Pi_N\tilde{\mathcal L}_\eps)(e):=\Pi_N\L_{T_{\eps, f}}e + \Pi_N \mathcal J(f)(e).
\end{align}
The continuity of the first term and its H\"older dependency on $f$ is a consequence of the Taylor expansion \eqref{eq:derivif},
\begin{align}\label{eq:contderivholder}
    \|\Pi_N\L_{T_{\eps, f}}e-\Pi_N\L_{T_{\eps, f+\tilde f}}e\|_{W^{1,1}}\leq \|(\Pi_N\partial_f\L_{T_{\eps,f}}e) [\tilde f]\|_{W^{1,1}}+(1+\|e\|_{W^{3,1}})o(\|\tilde f\|_{W^{1,1}}),
\end{align}
with $\partial_f\L_{T_{\eps,f}}e$ given by \eqref{eqderivif}
\begin{align*}   \partial_f(\L_{T_{\eps,f}}e)[\tilde f]=-\L_{I_f}\left[\left(\frac{\L_T e}{|I_{f}'|} \eps G \tilde f\right)'\right],
\end{align*}
using relations Lemmas \ref{lem:product} ,\ref{lem:LIfcon} and \ref{lem:boudifprimeinv}, one obtains the following bound
\begin{align*}
    \|(\Pi_N\partial_f\L_{T_{\eps,f}}e)[\tilde f]\|_{W^{1,1}}&\leq 2C_0\mathcal C\frac{R_2}{(1-\eps\|g\|_{C^1})^2}\|e\|_{W^{2,1}}\|\tilde f\|_{W^{1,1}}\\
    &\leq 2C_0 \mathcal C\frac{R_2}{(1-\eps\|g\|_{C^1})^2}C_N'\|e\|_{W^{1,1}}\|\tilde f\|_{W^{1,1}},
\end{align*}
using the fact that $e\in F_N$ and the equivalence of the norms $\|\cdot\|_{W^{1,1}}$ and $\|\cdot\|_{W^{2,1}}$, $\|\cdot\|_{W^{3,1}}$ and $\|\cdot\|_{W^{4,1}}$ on a finite-dimensional space, leading to the existence of a constant $C_N'>0$ 
such that $\|\cdot\|_{W^{4,1}}\leq C_N'\|\cdot\|_{W^{1,1}}$. Back to equation \eqref{eq:contderivholder} we thus have the continuity of the first term in the expression of the derivative $D_f(\Pi_N\tilde{\mathcal L}_\eps)(e)$ in \eqref{eq:descripderiv}:
\begin{align}\label{eq:contderivholder2}
    \|\Pi_N\L_{T_{\eps, f}}e-\Pi_N\L_{T_{\eps, f+\tilde f}}e\|_{W^{1,1}}\leq  2C_0 \mathcal C\frac{R_2}{(1-\eps\|g\|_{C^1})^2}C_N'\|e\|_{W^{1,1}}\|\tilde f\|_{W^{1,1}}+(1+C_N'\|e\|_{W^{1,1}})o(\|\tilde f\|_{W^{1,1}}).
\end{align}

We now establish the continuity of the remaining right hand term in \eqref{eq:descripderiv} whose explicit formula is given by formula \eqref{eqdeltafif}.
\begin{align}
  &\|  (\Pi_N\mathcal J(f)[e]-\Pi_N\mathcal J(f+\tilde f)[e] \|_{W^{1,1}}\nonumber\\
  &\leq \|\L_{I_f}\left[\left(\frac{f}{|I_{f}'|} \eps G e\right)'\right]-\L_{I_{f+\tilde f}}\left[\left(\frac{f+\tilde f}{|I_{f+\tilde f}'|} \eps G e\right)'\right]\|_{W^{1,1}}\nonumber\\
  &\leq \|\L_{I_f}\left[\left(\frac{\tilde f}{|I_{f}'|} \eps G e\right)'\right]-\L_{I_{f}}\left[\left(\frac{f+\tilde f}{|I_{f}'|} \eps G e\right)'\right]\nonumber\\
  &+\L_{I_{f}}\left[\left(\frac{f+\tilde f}{|I_{f+\tilde f}'|} \eps G e\right)'\right]-\L_{I_{f}}\left[\left(\frac{f+\tilde f}{|I_{f+\tilde f}'|} \eps G e\right)'\right]+\L_{I_{f+\tilde f}}\left[\left(\frac{f+\tilde f}{|I_{f+\tilde f}'|} \eps G e\right)'\right]\|_{W^{1,1}}\nonumber\\
   &\leq\|\L_{I_f}\left[\left(\frac{\tilde f}{|I_{f}'|} \eps G e\right)'\right]\|_{W^{1,1}}\label{eqderivderivun}\\
   &+\|\L_{I_{f}}\left[\left((f+\tilde f) \eps G e (\frac{1}{|I_{f}'|}-\frac{1}{|I_{f+\tilde f}'|} )\right)'\right]\|_{W^{1,1}}\label{eqderivderivdeu}\\
   &+\|\L_{I_{f}}\left[\left(\frac{f+\tilde f}{|I_{f+\tilde f}'|} \eps G e\right)'\right]-\L_{I_{f+\tilde f}}\left[\left(\frac{f+\tilde f}{|I_{f+\tilde f}'|} \eps G e\right)'\right]\|_{W^{1,1}}\label{eqderivderivtroi}.
\end{align}
Equation \eqref{eqderivderivun} can be dealt with directly:  using Lemma \ref{lem:LIfcon}, there is $C>0$ such that for any $f\in B_K$, 
\begin{align}\label{eq:boundLIf}
    \|\L_{I_f}\left[\left(\frac{\tilde f}{|I_{f}'|} \eps G e\right)'\right]\|_{W^{1,1}}\leq C\|\tilde f\|_{W^{2,1}}\|e\|_{W^{1,1}}.
\end{align}
For the second term \eqref{eqderivderivdeu} we use the fact that 
\begin{align*}
    \|\frac{1}{|I_{f}'|}-\frac{1}{|I_{f+\tilde f}'|}\|_{W^{2,1}}&=\|\frac{1}{I_{f}'}-\frac{1}{I_{f}'+\eps(G\tilde f)'}\|_{W^{2,1}}\\
    &=\|\frac{1}{I_{f}'}\sum_{k\geq 1}\left(\frac{\eps(G\tilde f)'}{I_f'}\right)^k\|_{W^{2,1}}\\
     &\leq 4\|\frac{1}{I_{f}'}\|_{W^{2,1}}\sum_{k\geq 1}(4\eps)^k\|\frac{(G\tilde f)'}{I_f'}\|_{W^{2,1}}^k\\
&\leq 4\|\frac{1}{I_{f}'}\|_{W^{2,1}}\sum_{k\geq 1}(16\eps)^k\|(G\tilde f)'\|_{W^{2,1}}^k\|\frac 1 {I_f'}\|_{W^{2,1}}^k\\
&\leq 4\|\frac{1}{I_{f}'}\|_{W^{2,1}}\sum_{k\geq 1}(\frac{16\eps R_2}{(1-\eps\|g\|_{C^1})^3})^k\|(G\tilde f)'\|_{W^{2,1}}^k\\
&\leq 4\|\frac{1}{I_{f}'}\|_{W^{2,1}}\sum_{k\geq 1}(\frac{16\eps R_2}{(1-\eps\|g\|_{C^1})^3})^k\|g\|_{C^3}\|\tilde f\|_{L ^1}^k\\
    &=O(\|\tilde f\|_{L^1}),
\end{align*}
where we have used Lemma \ref{lem:boudifprimeinv} in the third to last line above.
Since $\eps< \max\{\frac{1}{32R_2},\frac{1}{2^{1/3}\|g\|_{C^1}}\}$, the series in the penultimate line is an analytic series $z\mapsto \sum_{k\geq 1}\alpha_kz^k$ in $\|\tilde f\|_{L^1}$ with coefficient $\alpha_k$ less than $1$. 
Thus it behaves as $\|\tilde f\|_{L^1}O(1)$ for any $\tilde f\in F_N\cap B_{\|\cdot\|_{L^1}}(0,1)$.
Thus the term \eqref{eqderivderivdeu} can be controlled using Lemmas \ref{lem:product} and \ref{lem:LIfcon} as follows, there is $C>0$ such that
\begin{align}\label{eq:contderiv2}
    \|\L_{I_{f}}\left[\left((f+\tilde f) \eps G e (\frac{1}{|I_{f}'|}-\frac{1}{|I_{f+\tilde f}'|} )\right)'\right]\|_{W^{1,1}}&\leq 64\eps C_0\|f+\tilde f\|_{W^{2,1}}\|\frac{1}{|I_{f}'|}-\frac{1}{|I_{f+\tilde f}'|}\|_{W^{2,1}}\|Ge\|_{W^{2,1}}.\\
    &\leq C\|\tilde f\|_{L^1}\|f+\tilde f\|_{W^{2,1}}\|e\|_{L^1}.\nonumber
\end{align}
The last term \eqref{eqderivderivtroi} will be controlled by by the derivative formula computed in \eqref{eq:derivif}. Letting $H:=\left(\frac{f+\tilde f}{|I_{f+\tilde f}'|} \eps G e\right)'$, there is $\tilde C>0$ such that
\begin{align}
    \|\L_{I_{f}}\left[H\right]-\L_{I_{f+\tilde f}}\left[H\right]\|_{W^{1,1}}&\leq \|\partial_f\L_{I_f}[H]\|_{W^{1,1}}+ C\|\tilde f\|_{W^{1,1}}(1+\|H\|_{W^{3,1}})\nonumber\\
    &\leq C\|H\|_{W^{2,1}}\|\tilde f\|_{W^{1,1}}+ \eps C\|\tilde f\|_{W^{1,1}}\|f+\tilde f\|_{W^{4,1}}\|e\|_{W^{1,1}}\nonumber\\
    &\leq \tilde CC_N'\|\tilde f\|_{W^{1,1}}\|f+\tilde f\|_{W^{1,1}}\|e\|_{W^{1,1}},\label{eq:contderiv3}
\end{align}
where $\partial_f\L_{I_f}[H]$ is given by equation \eqref{eqderivif}, and estimated using formula \eqref{eq:boundLIf} with $h:=H$, and $\|H\|_{W^{k,1}}$ has been computed using Lemma \ref{lem:product} as follows :
\begin{align*}
    \|H\|_{W^{k,1}}&=\left\|\left(\frac{f+\tilde f}{|I_{f+\tilde f}'|} \eps G e\right)'\right\|_{W^{k,1}}\\
    &=2^{2(k+1)}\left\|f+\tilde f\right\|_{W^{k+1,1}}\left\|\eps G e\right\|_{W^{k+1,1}}\left\|\frac 1 {|I_{f+\tilde f}'|}\right\|_{W^{k+1,1}}\\
    &=2^{2(k+1)}\left\|f+\tilde f\right\|_{W^{k+1,1}}\eps\| g\|_{C^{k+1}} \|e\|_{W^{1,1}}\frac{R_{k+1}}{(1-\eps\|g\|_{C^1})^k}.
\end{align*}

Collecting everything, that is plugging \eqref{eq:boundLIf}, \eqref{eq:contderiv2}, \eqref{eq:contderiv3} and  the estimates above into \eqref{eqderivderivtroi}
we obtain that there is $\tilde C_N>0$ such that 
\begin{align*}
    \|  \Pi_N\mathcal J(f)[e]-\Pi_N\mathcal J (f+\tilde f)[e] \|_{W^{1,1}}\leq \tilde C_N \|\tilde f\|_{W^{1,1}}\|e\|_{W^{1,1}}.
\end{align*}
This along with the estimate \eqref{eq:contderivholder2} immediately implies the relation \eqref{eq:holderderiv} stated in the Lemma: there is $\mathcal C_N>C_N'>C_N^{(2)}$ such that
$\|D_f(\Pi_N\tilde{\mathcal L}_\eps)-D_{f+\tilde f}(\Pi_N\tilde{\mathcal L}_\eps)\|_{F_N\mapsto F_N}\leq \mathcal C_N\|\tilde f\|_{W^{1,1}}$.
\end{proof}

\subsection{Uniform boundedness of the resolvent}
\begin{proposition}
\label{prop:4.4} 
There is $\eps_7^*>0$ such that for $N$ large enough, there is  $C_{7,N}>0$ such that for $\eps\in [0,\eps_{7}^*]$ and all $f\in B_K\cap F_N$ one has $$\|(I-D_f(\Pi_N\tilde{\L_\eps}))^{-1}|_{F_N\cap V_{1,1}}\|_{W^{1,1}}\le C_{7,N}.$$ 
\end{proposition}

We prove Proposition \ref{prop:4.4} at the end of the section. 
First we need to state in Lemma \ref{lem:3.3} the existence of the resolvent $(I-D_f\tilde\L_{\eps,N})^{-1}$ for any $f\in B_K\cap F_N$, and to do this we use the following technical Lemma \ref{lem:uniform_bound_der}.
Recall, by Proposition \ref{prop:deriv}, we have
\begin{equation}\label{eq:discder}
    D_f(\tilde{\mathcal L}_{\eps,N})(e)=\Pi_N\L_{T_{\eps, f}}e -\eps \tilde D_{N,f,G}(e),
\end{equation}
where
$$
   \tilde D_{N,f,G}(e):= \Pi_N\L_{I_f}\left[\left(\frac{\L_{T}(f)\cdot G(e)}{|I_{f}'|}\right)'\right].
$$
\begin{lemma}\label{lem:uniform_bound_der}
For $\eps<\frac{1}{\|g\|_{C^1}}$, 
    $\exists C>0$ such that for all $f\in B_K$ and $e\in F_N$,
    $$\|\tilde D_{N,f,G}(e)\|_{W^{1,1}}\le \frac{CK}{(1-\eps\|g\|_{C^2})^{3}} \|g\|_{C^1}\|e\|_{L^1}.$$
\end{lemma}
\begin{proof}
Since $\L_{I_f}$ is continuous on $W^{1,1}$ (see Lemma \ref{lem:LIfcon})  
by 
Lemma \ref{lemkern}, one can find a constant $C$ such that
$$\|\tilde D_{N,f,G}(e)\|_{W^{1,1}}=\left\|\Pi_N\L_{I_f} \left[\left(\frac{\L_{T}(f)\cdot G(e)}{|I_{f}'|}\right)'\right]\right\|_{W^{1,1}}\le C\left\| \left[\left(\frac{\L_{T}(f)\cdot G(e)}{|I_{f}'|}\right)'\right]\right\|_{W^{1,1}}.$$ 
Therefore it is enough to bound 
$\left\|\left(\frac{\L_{T}(f)\cdot G(e)}{|I_{f}'|}\right)'\right\|_{W^{1,1}}$.

One has using Lemma \ref{lem:product} and Lemma \ref{lem:boudifprimeinv},
\begin{align}\label{eq:boundnumif}
\left\|\left(\frac{\L_Tf}{|I_{f}'|}  G e\right)'\right\|_{W^{1,1}}&\leq \left\|\frac{\L_Tf}{|I_{f}'|} G e\right\|_{W^{2,1}}\nonumber\\
&\leq \frac{2R_3K}{(1-\eps\|g\|_{C^2})^3} \|g\|_{C^2}\|e\|_{L^1},
\end{align}
using the facts that since $f\in B_K$, $\|f\|_{W^{2,1}}\leq K$ and that $B_K$ is $\L_T$-invariant.
\end{proof}

\begin{lemma}
\label{lem:3.3}
There is $\eps^*_6>0$ and $N_6\in \mathbb N$ such that for any  $0<\eps<\eps^*_6$, any  $N\ge N_6$, and any $f\in B_K$,
if $\rho$ is an 
eigenvalue of $D_f(\Pi_N\tilde{\mathcal L}_\eps)$ 
corresponding to an eigenfunction $e$ with $\int e=0$, then $|\rho|< 1$.
\end{lemma}
\begin{proof}
We first show that for $\eps$ and $N$ in the prescribed ranges,
$\Pi_N\L_{T_{\eps,f}}$ admits a spectral gap on $W^{i,1}$, $i=1,\dots,4$, uniformly in $\eps$.
Notice that by \eqref{eq:standLY} there is some $\eps_5^*>0$, for $T_{\eps,f}\in\mathcal T$, $\L_{T_{\eps,f}}$ admits a spectral gap on $W^{i,1}$, uniform in $f\in B_K$ and in $\eps$ for $0<\eps<\eps^*_5$ and $i=1,\dots,4.$
By the first item of Lemma \ref{lemkern}, the operator $\Pi_N\L_{T_{\eps,f}}$ satisfies the same Lasota-Yorke inequality as $\L_{T_{\eps,f}}$ for all $N\in\mathbb{N}$. Moreover, the second item of Lemma \ref{lemkern} implies, by the Keller-Liverani stability result \cite{KL99}, that there is an $N_6^*$  such that for $N>N_6^*$ and $f\in B_K$, one has $\Pi_N\L_{T_{\eps,f}}$ also admits a spectral gap on $W^{i,1}$, $i=1,\dots,4.$ 
We fix $N$ in this range for the remainder of the proof.

By \eqref{eq:discder}, if $\rho$ is an eigenvalue of $D_f(\Pi_N\tilde{\mathcal L}_\eps)$, then

\begin{align}\label{eq:eig}
    \rho e=\Pi_N\L_{T_{\eps, f}}e -\eps\tilde D_{N,f,G}(e).
\end{align}
Note that if $\rho$ were in $\sigma(\Pi_N\L_{T_{\eps, f}})$ then we must have $|\rho|<1$ because $\int e=0$ (using spectral decomposition of $\Pi_N\L_{T_{\eps, f}}$), and we would be done.
We now treat the situation where $\rho\notin\sigma(\Pi_N\L_{T_{\eps, f}})$. 

For a contradiction, assume that $|\rho|\ge1$.
We can also assume that $\|e\|_{W^{1,1}}=1$ simply by normalising. Then using \eqref{eq:eig}, we get 
$$e=-\eps(\rho\text{id}-\Pi_N\L_{T_{\eps, f}})^{-1}\left(\tilde D_{N,f,G}(e)\right)$$
and consequently, since $\int\tilde D_{N,f,G}(e)=0$, and using the contradiction hypothesis,
\begin{equation}
    \label{contrinequality}
\|e\|_{W^{1,1}}=1\le \eps\sum_{n=0}^\infty\frac{\|(\Pi_N\L_{T_{\eps, f}})^{n}\left(\tilde D_{N,f,G}(e)\right)\|_{W^{1,1}}}{|\rho|^{n+1}}\le \eps \frac{C}{1-\theta}\|\tilde D_{N,f,G}(e)\|_{W^{1,1}},
\end{equation}
where we have used the fact that $\Pi_N\L_{T_{\eps, f}}$ admits a spectral gap on $W^{1,1}$ and therefore $\exists C>0$, and $\theta\in(0,1)$  such that $\|(\Pi_N\L_{T_f})^{n}\left(\tilde D_{N,f,G}(e)\right)\|_{W^{1,1}}\le C\theta^n\|\tilde D_{N,f,G}(e)\|_{W^{1,1}}$.
Since for $\eps<\frac{1}{\|g\|_{C^1}}$ and $f\in B_K$, $\|\tilde D_{N,f,G}(e)\|_{W^{1,1}}$ is uniformly bounded in $f$, see Lemma \ref{lem:uniform_bound_der}.
Consequently  there is an $\eps_6^*$ such that \eqref{contrinequality} does not hold for any $\eps<\eps_6^*\le \eps_5^*$, leading to a contradiction.
Therefore $|\rho|<1$.
\end{proof}
\begin{proof}[Proof of Proposition \ref{prop:4.4}] 
 Let $\eps_7^*:=\min\{\frac{1}{32R_2},\frac{1}{2^{1/3}\|g\|_{C^1}},\eps_6^*\}$. Lemma \ref{lem:3.3} ensures that for $\eps<\eps_7^*$ and $N>N_6$, $(I-D_f(\Pi_N\tilde{\L_\eps}))^{-1}$ is well defined operator when acting on $V_{1,1}$.
 By Lemma \ref{lemcontinuityderiv}, the map $f\in B_K\cap F_N\mapsto D_f\Pi_N\tilde \L_\eps \in L(F_N,F_N)$ is continuous in {the $W^{1,1}$ topology}, and since $F_N$ is a finite-dimensional space,  the map $f\in B_K\cap F_N\mapsto (I-D_f\Pi_N\tilde \L_\eps)^{-1}$ is also continuous {on $V^{1,1}$.}
 Indeed, the proof of continuity using Neumann series is standard: for any $g\in B_K\cap F_N$ close to $f$ such that one has 
 $\|D_f\Pi_N\tilde \L_\eps-D_g\Pi_N\tilde \L_\eps\|_{V^{1,1}}< \|(I-D_f\Pi_N\tilde \L_\eps)^{-1}\|_{V^{1,1}}$,
\begin{align*}
    &\|(I-D_g\Pi_N\tilde \L_\eps)^{-1}\|_{V^{1,1}}\leq \|(I-D_f\Pi_N\tilde \L_\eps+D_f\Pi_N\tilde \L_\eps-D_g\Pi_N\tilde \L_\eps)^{-1}\|_{V^{1,1}}\\
     &\leq \|(I-D_f\Pi_N\tilde \L_\eps)^{-1}\left[I-(I-D_f\Pi_N\tilde \L_\eps)^{-1}(D_g\Pi_N\tilde \L_\eps-D_f\Pi_N\tilde \L_\eps)\right]^{-1}\|_{V^{1,1}}\\
     &\leq \|(I-D_f\Pi_N\tilde \L_\eps)^{-1}\|_{V^{1,1}}\sum_{k=0}^\infty\|(I-D_f\Pi_N\tilde \L_\eps)\|_{V^{1,1}}^k\|(D_g\Pi_N\tilde \L_\eps-D_f\Pi_N\tilde \L_\eps)\|_{V^{1,1}}^k,
\end{align*}
Since $F_N\cap B_K$ is a compact set and $f\in F_N\cap B_K\mapsto \|(I-D_f\Pi_N\tilde \L_\eps)^{-1}\|_{V^{1,1}}$ is continuous, it attains a maximum and there is $C_{7,N}>0$ such that for any $f\in B_K\cap F_N$,
 \begin{align*}
    \|(I-D_f\Pi_N\tilde \L_\eps)^{-1}\|_{V^{1,1}}\leq C_{7,N}.
 \end{align*}

\end{proof}

\section{Numerical implementation of Newton's method}\label{sec:num}

From \eqref{eq:newtonupdate} and the definition of $\mathcal{N}$ we see that the application of Newton's method primarily requires computing the action of $\Pi_N\tilde{\L}_{\eps}$ and $(\Pi_N(I-D_f(I-\Pi_N\tilde{\L}_{\eps}))^{-1}$.
Equations \eqref{fullderiveqn} and \eqref{eqdeltafif} provide an expression for the Fr\'echet derivative of $\Pi_N\tilde{\L_\eps}$, namely:
\begin{eqnarray} 
&&\nonumber D_f\Pi_N\tilde{\mathcal L}_\eps(e)=\Pi_N\left(\L_{T_{\eps,f}}e -\eps\L_{I_f}\left[\left(\frac{\L_{T}(f)\cdot G(e)}{|I_{f}'|}\right)'\right]\right)\\
\label{eq:derivexpanded}&&=\Pi_N\L_{T_{\eps,f}}e -\eps \Pi_N\L_{I_f}\left[\frac{\L_{T}(f)'\cdot G(e)}{I_{f}'} + \frac{\L_{T}(f)\cdot G(e)'}{I_{f}'} - \frac{\L_{T}(f)\cdot G(e)\cdot I''_f}{(I_{f}')^2}\right].
\end{eqnarray}
Thus, the main numerical tasks are estimating each of the terms in \eqref{eq:derivexpanded}.

{To implement the operator $D_f\Pi_N\tilde{\mathcal L}_\eps$  numerically we first need to approximate it by 
\begin{align*}
    &M_{N,f,\eps}:= \Pi_N\L_{T_{\eps,f}}e\\
    &-\eps \Pi_N\L_{I_f}\left[\frac{\Pi_N(\L_{T}(f))'\cdot \Pi_N G(e)}{\Pi_N I'_f} + \frac{\Pi_N(\L_{T}(f))\cdot (\Pi_NG(e))'}{\Pi_N I_{f}'} - \frac{\Pi_N\L_{T}(f)\cdot \Pi_NG(e)\cdot (\Pi_N I'_f)'}{(\Pi_N I_{f}')^2}\right].\nonumber
\end{align*}
We first approximate the error made when working with $M_{N,f,\eps}$ instead of $D_f\Pi_N\tilde{\mathcal L}_\eps$.
We compare $D_f\Pi_N\tilde{\mathcal L}_\eps$ and $M_{N,f,\eps}$ on $L(F_N,F_N)$, for $f\in B_K$ and $e\in F_N$ with $\|e\|\leq 1$.
\begin{align*}
    &M_{N,f,\eps}(e)-D_f\Pi_N\tilde{\mathcal L}_\eps(e)\\
    &\hskip 0.5cm=\eps \Pi_N\L_{I_f}\left[\frac{\Pi_N(\L_{T}(f))'\cdot \Pi_N  G(e)}{\Pi_N I'_f}-\frac{\L_{T}(f)'\cdot G(e)}{I_{f}'} + \frac{\Pi_N(\L_{T}(f))\cdot (\Pi_NG(e))'}{\Pi_N I_{f}'}\right.\\
    &\hskip 0.75cm\left. - \frac{\L_{T}(f)\cdot G(e)'}{I_{f}'} -  \frac{\Pi_N\L_{T}(f)\cdot \Pi_NG(e)\cdot (\Pi_N I'_f)'}{(\Pi_N I_{f}')^2}+\frac{\L_{T}(f)\cdot G(e)\cdot I''_f}{(I_{f}')^2}\right]\\
    &\hskip 0.5cm=\eps \Pi_N\L_{I_f}\left[ (I)+(II)+(III) \right].
\end{align*}
Note that $(I), (II), (III)$ can be easily estimated using Lemma \ref{lemkern}. We provide a detailed estimate for $(I)$. The estimates for $(II)$ and $(III)$ obtained similarly. Below we repeatedly use $B_K$ is invariant under $\L_T$, and therefore $\|\L_T(f)'\|_{W^{1,1}}\le K$.
\begin{align*}
  &\|(I)\|_{L^1}\leq  \left\| \frac{\Pi_N(\L_{T}(f))' \cdot \Pi_N G(e)}{\Pi_N I'_f}-\frac{\L_{T}(f)'\cdot G(e)}{I_{f}'}\right\|_{L^1}\\
  \leq & \left\| \frac{(\Pi_N(\L_{T}(f))'-\L_{T}(f)') \cdot \Pi_NG(e)}{\Pi_N I'_f}+\frac{\L_{T}(f)'\cdot (\Pi_N-I)G(e)}{\Pi_NI_{f}'}\right.\left.+\L_{T}(f)'\cdot G(e) \left(\frac{1}{\Pi_NI_{f}'}-\frac{1}{I_{f}'}\right)\right\|_{L^1}\\
  \leq & \left\| \frac{(\Pi_N(\L_{T}(f))-\L_{T}(f))' \cdot \Pi_NG(e)}{\Pi_N I'_f}\right\|_{L^1}\\
  &+\left\|\frac{\L_{T}(f)'\cdot (\Pi_N-I)G(e)}{\Pi_NI_{f}'}\right\|_{L^1}+\left\|\L_{T}(f)'\cdot G(e) \left(\frac{1}{\Pi_NI_{f}'}-\frac{1}{I_{f}'}\right)\right\|_{L^1}\\
  \leq & 2K\eps \frac{\|g\|_{C^1}\|e\|_{L^1}}{1-\eps\|g\|_{C^2}\|e\|_{L^1}} O(\frac{\ln(N)}{N})+ K\eps \|g\|_{C^2}\|e\|_{L^1}\left\| \frac{(\Pi_N-I)I_{f}'}{(\Pi_NI_{f}')(I_{f}')}\right\|_{L^1}\\
\leq & 2K \eps\frac{\|g\|_{C^1}\|e\|_{L^1}}{1-\eps\|g\|_{C^2}\|f\|_{L^1}} O(\frac{\ln(N)}{N})+ \frac{K \eps\|g\|_{C^2}\|e\|_{L^1}(1+\eps\|g\|_{C^2}\|f\|_{L^1})}{(1-\eps \|g\|_{C^2}\|f\|_{L^1})^2}O(\frac{\ln(N)}{N})\\
= &O\left(\frac{\ln(N)}{N}\right)\|e\|_{L^1},
\end{align*}
where the third-last line uses a similar bound as Equation \eqref{eq:boundnumif}. 
Therefore, 
\begin{align}\label{eq:Mdiffdif}
    \| M_{N,f,\eps}e-D_f\Pi_N\tilde{\mathcal L}_\eps e\|_{L^1}=O\left(\frac{\ln(N)}{N}\right)\|e\|_{L^1}.
\end{align}
Finally, we define $D_{N,f}$, which in comparison to $M_{N,f,\eps}$ has additional $\Pi_N$ on each of the inner square bracket terms to truncate any modes higher than order $N$ obtained through multiplication or division.
\begin{align*}
   D_{N,f}e :=& \Pi_N\L_{T_{\eps,f}}e-
    \eps \Pi_N\L_{I_f}\left[\Pi_N\left[\frac{\Pi_N(\L_{T}(f))' \cdot \Pi_NG(e)}{\Pi_N I'_f}\right]\right.\\
    &\left.+ \Pi_N\left[\frac{\Pi_N(\L_{T}(f))\cdot (\Pi_NG(e))'}{\Pi_N I_{f}'}\right] - \Pi_N\left[\frac{\Pi_N\L_{T}(f)\cdot \Pi_NG(e)\cdot (\Pi_N I'_f)'}{(\Pi_N I_{f}')^2}\right]\right].\nonumber
\end{align*}
We now show that running Newton's method with $D_{N,f}$ instead of the true derivative $D_f\tilde{\mathcal{L}}_{\eps,N}$ still provides an exponential rate of convergence to the fixed point of $\tilde{\mathcal{L}}_{\eps,N}$.
\begin{proposition}
For $\eps\in (0,\eps_6^*)$ and $N\in \mathbb N$ large enough, the sequence $(h^n)_{n\in \mathbb N}\subset B_K\cap F_N$ computed through the Newton iteration
\begin{align*}
    h^{n+1}=h^n-(I-D_{N,h^n})^{-1}(I-\tilde \L_{\eps,N})h^n,
\end{align*}
 converges at exponential rate for $h^0$ close enough to $h_N^*$.
\end{proposition}

\begin{proof}
A straightforward application of Lemma \ref{lemkern} and Lemma \ref{lem:product} implies that for all $f\in B_K\cap F_N$ and $e\in F_N$, 
\begin{align*}
    \|D_{N,f}e- M_{N,f,\eps}e\|_{L^1}&\leq \|\eps \Pi_N\L_{I_f}\left[(\Pi_N-I)\left[\left[\frac{\Pi_N(\L_{T}(f))' \cdot \Pi_NG(e)}{\Pi_N I'_f}\right]\right.\right.\\
    &\left.\left.+ \Pi_N\left[\frac{\Pi_N(\L_{T}(f))\cdot (\Pi_NG(e))'}{\Pi_N I_{f}'}\right] - \Pi_N\left[\frac{\Pi_N\L_{T}(f)\cdot \Pi_NG(e)\cdot (\Pi_N I'_f)'}{(\Pi_N I_{f}')^2}\right]\right]\right]\|_{L^1}\\
    =&\eps C \cdot O\left(\frac{\ln(N)}{N}\right) \left[ \frac{K}{1-\eps\|g\|_{C^2}\|f\|_{L^1}}\|G(e)\|_{W^{1,1}}\right.\\
    &\left.+ \frac{K}{1-\eps\|g\|_{C^2}\|f\|_{L^1}}\|\Pi_N(G(e))'\|_{W^{1,1}}+\frac{K(1+\eps\|g\|_{C^3}\|f\|_{L^1})}{(1-\eps\|g\|_{C^2}\|f\|_{L^1})^2}\|G(e)\|_{W^{1,1}}\right]\\
    = &O\left(\frac{\ln(N)}{N}\right)3\eps C \frac{K(1+\eps\|g\|_{C^3}\|f\|_{L^1})}{(1-\eps\|g\|_{C^2}\|f\|_{L^1})^2}\|g\|_{C^2}\|e\|_{L^1}\\
    =&O\left(\frac{\ln(N)}{N}\right)\|e\|_{L^1},
\end{align*}
where we used similar bound as in Equation \eqref{eq:boundnumif} in the third to last line.
and thus using the above along with inequality \eqref{eq:Mdiffdif}, for all $f\in B_K\cap F_N$ and $e\in F_N$,
\begin{align*}
    \| D_f\Pi_N\tilde{\mathcal L}_\eps e-D_{N,f}e\|_{L^1}=O\left(\frac{\ln(N)}{N}\right)\|e\|_{L^1}.
\end{align*}

 This combined with equation \eqref{eq:Frechetcondition} with norm $\|\cdot\|$ standing for $\|\cdot\|_{W^{1,1}}$ from Section \ref{sec:newtheory}, which is guaranteed by equation \eqref{eq:Frechetderivreg} (which is itself implied by Lemma \ref{lemcontinuityderiv}), 
  implies for $N$ large enough and $e^n\in F_N$ small enough (depending on $N$), there exists $\gamma'\in(0,1)$ such that,
\begin{align*}
     \|\mathcal{N}(h^*_N+e^n)-\mathcal{N}(h^*_N)-(I-D_{N,f})e\|_{L^1}\le& \|\mathcal{N}(h^*_N+e^n)-\mathcal{N}(h^*_N)-D_{h^*_N+e^n}\mathcal{N}(e^n)\|_{W^{1,1}}\\
     &+  \|(I-D_{N,f})e^n- D_{h^*_N+e^n}\mathcal{N}(e^n)\|_{L^1}\\
     \leq &C_N\|e^n\|_{W^{1,1}}^{2}+O\left(\frac{\ln(N)}{N}\right)\|e^n\|_{L^1}\\
    \leq &C_NC_N'\|e^n\|_{L^1}^{2}+O\left(\frac{\ln(N)}{N}\right)\|e^n\|_{L^1}\\
     \leq &\gamma'\|e^n\|_{W^{1,1}},
\end{align*}
where $C_N'$ stands for the equivalence bound between $\|\cdot\|_{L^1}$ and $\|\cdot\|_{W^{1,1}}$ on $F_N$ that is the element $C_N'$ such that for any $e\in F_N$, $\|e\|_{W^{1,1}}\leq C_N'\|e\|_{L^1}$.
Consequently, according to Section 2.4.2 from \cite{hinze2008optimization}, in particular Part $1$ of Theorem 2.9 \cite{hinze2008optimization} and the comment immediately following the proof of Theorem 2.9, the current Newton's method with our current estimate $D_{N,f}$ converges at exponential speed toward $h_N^*$.
\end{proof}

Referring to \eqref{eq:derivexpanded}, if one were to use the less natural alternative coupling $T\circ I_f$ instead of $I_f\circ T$ we use, one would remove $\L_T$ inside the square brackets and replace $\L_{I_f}$ with $\L_{T_f}$. 
In principle one could then evaluate the resulting terms in the square brackets in \eqref{eq:derivexpanded} to an ``arbitrary'' precision (up to a tolerance larger than machine precision) with numerical quadrature instead of FFT. 
Alternatively, if $g$ has Fourier modes restricted to order $N$ or less, then one could exactly evaluate the objects in the square brackets with FFT (again up to FFT tolerance).
Either option would then eliminate the need to consider $D_{N,f}$ and the above associated error estimates, and instead require keeping track of numerical quadrature or FFT tolerances.

}

\subsection{Numerical construction of $\Pi_N\circ\L_{T}$}
\label{sec:numLT}
A basic building block for the numerical implementation is the construction of $\Pi_N\circ\L_{T}$, acting on the Fourier basis $F_N:=\{\exp(2\pi i k x): k\in\{-N +1,\ldots,-1,0,1,\ldots,N\}$.
An almost identical construction will be used to form $\Pi_N\circ\L_{T_{\eps,f}}$ and $\Pi_N\circ\L_{I_f}$.
We denote $e_k(x)=\exp(2\pi i k x)$ and write
 $h=\sum_{k=-N+1}^N \hat h(k) e_k$.
Whenever we compute a Fourier transform, it is carried out by fast Fourier transform (FFT) \cite{FFTW} on a grid on $S^1$ of size $16N$ to ensure satisfactory accuracy of the Fourier integrals to order $N$.
Whenever we revert from frequency space back to physical space on $S^1$ we construct a sum as above for $h$, which is typically then sampled on the $16N$ points.

We denote by $\hat L_N$ the action on the Fourier coefficients of $h$ induced by the action of $\Pi_N\L_{T}$ on $h$.
Recalling the definition of $K_N$ from \S\ref{sec:disc} one has
\begin{eqnarray}
\nonumber    \widehat{\left((\Pi_N\circ\L_{T})h\right)}(k) &=& \langle \Pi_N\circ(\L_{T}h), e_k\rangle \\
\nonumber     &=& \hat{K}_N(k)\cdot\langle h, e_k\circ T\rangle\\
  \nonumber   &=&\hat{K}_N(k)\sum_{i=-N+1}^{N} \hat h(i)\langle e_i, e_k\circ T\rangle\\
  \nonumber      &=&\hat{K}_N(k)\sum_{i=-N+1}^{N} \hat h(i)\langle e_{-k}\circ T, e_{-i}\rangle\\
\nonumber   &=&\hat{K}_N(k)\sum_{i=-N+1}^{N} \hat h(i)\cdot \widehat{(e_{-k}\circ T)}(-i)\\
\label{eq:TOdiscfinal}    &=&\sum_{i=-N+1}^{N} \underbrace{\hat{K}_N(k)\cdot \widehat{(e_{-k}\circ T)}(-i)}_{=:\hat L_{N,ki}}\cdot\hat h(i).
\end{eqnarray}
This construction requires $2N$ FFTs of $e_k\circ T$, for $k=-N+1,\ldots,N$.
Note that we have an explicit expression for $\hat{K}_N(k)$.

In our numerical experiments in Section \ref{sec:examples}, both our sequential iteration scheme and Newton iteration scheme will be initialised at $h_{0,N}^*$, whose Fourier coefficients are computed as the leading right eigenvector of the $2N\times 2N$ matrix $\hat{L}_N$ defined in \eqref{eq:TOdiscfinal}.

\subsection{Numerical construction of $\Pi_N\circ\L_{T_{\eps,f}}$}
\label{sec:LTf}

To obtain an estimate of $\Pi_N\L_{T_{\eps,f}}$ we follow the construction above for $\Pi_N\L_{T}$, replacing $T$ with an estimate of $T_{\eps,f}$.
Note that $T_{\eps,f}$ will arise from numerical estimates of $f$, which live in the span of $F_N$.
To simplify the numerical computations, we assume that the coupling function is of the form $g(x,y)=g(x-y)$.
Using \eqref{eq:couplingmap}, \eqref{eq:coupleddy}, and the form of $g$, we have $T_{\eps,f} = T + \eps((g\circ T)\ast f)$, where $\ast$ denotes convolution.
To compute $(g\circ T)\ast f$ we calculate $\widehat{(g\circ T)\ast f}=\widehat{(g\circ T)}\cdot\hat{f}$. 
We then estimate 
\begin{equation}
    \label{eq:Tfest}
T_{\eps,f}(x)\approx T(x) + \eps\sum_{k=-N+1}^{N}\left(\widehat{(g\circ T)}(k)\cdot\hat{f}(k)\right)\cdot e_k(x).
\end{equation}
The calculation of $\widehat{g\circ T}$ only has to be done once and is fixed forever, and at a given Newton iteration, we already have $\hat{f}$ for the current $f$ in the iteration.
We now insert this estimate for $T_{\eps,f}$ 
into \eqref{eq:TOdiscfinal} to compute $\widehat{e_{-k}\circ T_{\eps,f}}$ for $k=-N+1,\ldots, N$.
Thus, every computation of $\Pi_N\circ\L_{T_{\eps,f}}$ requires $2N$ FFTs.

\subsection{Numerical construction of $\Pi_N\circ\L_{I_f}$}

We again follow the construction above for $\Pi_N\L_{T}$, replacing $T$ with $I_f$.
We have $I_f = \mathrm{Id} + \eps(g\ast f)$, so \eqref{eq:Tfest} becomes
\begin{equation}
    \label{eq:Ifest}
I_{f}(x)\approx x + \eps\sum_{k=-N+1}^{N}\hat{g}(k)\cdot\hat{f}(k)\cdot e_k(x).
\end{equation}
The computation of $\hat g$ can be done once and is fixed forever.
We now insert this estimate for $I_{f}$ 
into \eqref{eq:TOdiscfinal} to compute $\widehat{e_{-k}\circ I_{f}}$ for $k=-N+1,\ldots, N$.
Thus, every computation of $\Pi_N\circ\L_{I_{f}}$ requires $2N$ FFTs.

\subsection{Computation of the terms in the Fr\'echet derivative}

Referring to \eqref{eq:derivexpanded}, we have already described how to estimate $\Pi_N\L_{T_{\eps,f}}, \Pi_N\L_{I_f}$, and $\Pi_N\L_T$.
In this section we detail how we estimate the derivative of $\Pi_N\L(f)$, the derivatives of $I_f$, and the functions $G(e_k)$ and their derivatives.

\subsubsection{Derivative of $\Pi_N\L_T(f)$}
The term $\Pi_N\L_T(f)'$ is estimated by first computing $\Pi_N\L_T(f)$, which amounts to multiplying the matrix $\hat{L}_N$ by the Fourier coefficients of $f$.
For $f\in F_N$, this yields $(\Pi_N\L_T)(f)=\sum_{k=-N+1}^{N} (\hat{L}_N\hat f)_k\cdot e_k$.
To take the spatial derivative we write $$(\Pi_N\L_T)(f)'=\sum_{k=-N+1}^{N} (2\pi k i)(\hat{L}_N\hat f)_k\cdot e_k.$$

\subsubsection{Derivatives of $I_f$}
The derivatives of $I_f$ are treated as follows.
One has $I'_f=1+\eps(g'\ast f)$. Thus $\widehat{I'_f}(k)=\delta_{0k} + \eps \cdot (2\pi ik)\hat{g}(k)\cdot \hat{f}(k)$. 
The approximate spatial representation of $I'_f$ is then assembled as
$I'_f(x)\approx 1+\sum_{k=-N+1}^{N} \widehat{I'_f}(k)e_k(x)$.

Similarly, $I''_f=\eps(g''\ast f)$ and so $\widehat{I''_f}(k)=\eps \cdot (2\pi ik)^2\hat{g}(k)\cdot \hat{f}(k)$. 
The approximate spatial representation of $I''_f$ is assembled as above.

\subsubsection{$G(e_k)$ and its derivative}

In the Newton iteration we require an estimate of the operator $\Pi_N\circ D\tilde\L|_f$, represented in the basis $F_N$.
Thus we require the action of $G(e)(x)=\int g(x-y)e(y)\ dy$ in the RHS of \eqref{eq:derivexpanded} on each $e\in F_N$, and in particular on each basis element $e_k$.
Because we require estimates of $G(e)$ only for $e=e_k$, the terms $g\ast e_k$ can be dealt with as follows:  $\widehat{g\ast e_k}(j)=\hat g(j)\cdot \hat e_k(j)=\hat g(j)\cdot \delta_{kj}$. We may Fourier transform $g$ once and store it (in fact, this transform has already been computed for creating $\Pi_N\L_{I_f}$, see \eqref{eq:Ifest}).
We recover $G(e_k)\approx \sum_{j=-N+1}^{N} \hat g(j)\delta_{kj}e_k = \hat g(k)e_k.$
The spatial derivative $G(e_k)'$ is thus estimated as $2\pi ki \cdot \hat g(k)e_k$.

\subsection{Combining the terms}
\label{sec:combining}
Having constructed spatial estimates of all of the individual terms in the square brackets of \eqref{eq:derivexpanded} for a given $k$ (recall we estimate $G(e_k)$ one $k$ at a time), one simply multiplies/divides/adds/subtracts as appropriate to estimate the combination in the square brackets, let us temporarily call it $\mathrm{Comb}_{N,k}$.
We now need to return to frequency space, so we Fourier transform this combination
and calculate (as per \eqref{eq:derivexpanded} in frequency space)
\begin{equation}
\widehat{D_f\Pi_N\tilde\L_{\eps}(e_k)}=\widehat{\Pi_N\L_{T_{\eps,f}}(e_k)}-\eps\widehat{\Pi_N\L_{I_f}}(\widehat{\mathrm{Comb}_{N,k}}).\nonumber
\end{equation}
For each $k=\{-N+1,\ldots,N\}$ we place the resulting $2N$-vector in the $k+N$th column of a $2N\times 2N$ matrix called $\hat D_{N,f}$, representing the Fr\'echet derivative in frequency space.

\subsection{Newton iteration}\label{subsecnewtfourier}

We are now ready to implement the Newton iteration in frequency space.
We initialise Newton with $\hat{h}_{0,N}^*$, obtained as described at the end of Section \ref{sec:numLT}.
Because ${h}_{0,N}^*$ is a density, we know that $\hat{h}_{0,N}^*(0)=1$.
Further, since we wish to remain in the class of densities, it will be convenient to restrict the Newton iteration to the subspace $\bar{F}_N=\{e_k: k\in \{-N+1,\ldots,-2,-1,1,2,\ldots, N\}\}\subset F_N$ spanned by nonconstant Fourier modes, or equivalently nonzero frequencies.

Our Newton update is:
\begin{equation}
    \label{eq:newtonupdate2}
    \hat{h}_N^{n+1} := \hat{h}_N^n - (I-\hat{D}_{N,h_N^n})^{-1}\cdot (\hat h_N^{n} - \hat{L}_{N,h_N^n}\hat h_N^n),
\end{equation}
where $\hat D_{N,h_N^n}$ is constructed as in Section \ref{sec:combining} using $f=h_N^n$, and $\hat L_{N,h_N^n}$ is constructed as in Section \ref{sec:LTf} using $f=h_N^n$.
Equation \eqref{eq:newtonupdate2} is to be understood as computed over the nonzero frequencies.
While in theory, the constant mode of $\hat h_N^{n} - \hat{L}_{N,h_N^n}\hat h_N^n$ should be zero, because of small machine-precision errors, it is more accurate to simply do the computation in \eqref{eq:newtonupdate2} over the nonzero frequencies.

\subsection{Computational cost}
At each Newton iteration, we have a new fixed-point estimate $h^n_N$ and require new estimates of $\L_{T_{\eps,h^n_N}}$ and $\L_{I_{h^n_N}}$, costing $2N$ FFTs each over a finer spatial grid of size $16N$.
The Fr\'echet derivative also changes at each iteration and the $2N\times 2N$ matrix $I-\hat D_{N,h^n_N}$ must be inverted.
There are also some reconstruction costs when converting from frequency space back to physical space.
FFTs are generally extremely fast to evaluate and at modest values of $N$ around 100, inverting $I-\hat D_{N,h^n_N}$ is very cheap.
In contrast, sequential iteration requires half the number of FFTs per iteration because one only constructs a single approximate operator $\L_{T_{\eps,h^n_N}}$, namely $\hat L_{N,h^n_N}$, to apply to the current fixed-point estimate.

In terms of advantages of Newton iteration over sequential iteration, we note that sequential iteration can only find attracting fixed points while Newton iteration is also able to find repelling fixed points.
As we show in the next section, the modest per-iteration computational cost increase of Newton,  relative to sequential iteration, is offset by a significantly faster convergence rate.
This greatly reduces the total number of iterations required and leads to an overall computational cost improvement with Newton.

\section{Examples}
\label{sec:examples}

Consider the uncoupled map $T_0:S^1\to S^1$ defined by $T_0(x)=2x-(a/2\pi)\sin(2\pi x)\pmod 1$, with $a=0.9$.
This two-branched map of the circle has relatively low derivative $T_0'(0)=1.1$ at $x=0$, leading to trajectories spending a greater amount of time near the ``sticky'' fixed point $x=0$.
The invariant density of $T_0$ describes the long-term distribution of trajectories in $S^1$ and is shown as a dashed line in Figure \ref{fig:translation}(centre), with a clear peak at $x=0$.
This invariant density has been estimated as the fixed point $h_{0,N}^*$ of $\Pi_N\mathcal{L}_{T}$ and computed as the leading eigenvector of $\Pi_N\mathcal{L}_T$ with $N=256$.

Let $b:S^1\to [0,1]$ be a $C^\infty$ bump function given by 
$$b(x)=\begin{cases}\exp(1)\cdot\exp\left(\frac{1}{\left(\left(x - \frac{1}{2}\right) / \delta\right)^2 - 1}\right)&\mbox{if $\left|x-\frac{1}{2}\right|\le\delta$,}\\
0 &\mbox{otherwise.}\end{cases}$$
with $\delta=0.45$.
We investigate two kernels $g$ of the form $g(x,y)=g(x-y)$, where (i) $g(x)=b(x)$ and (ii) $g(x)=-b'(x)/5$;  we apply the factor 1/5 to the latter kernel so that its range lies approximately in $[-1,1]$.
\begin{figure}
    \centering
    \includegraphics[width=\linewidth]{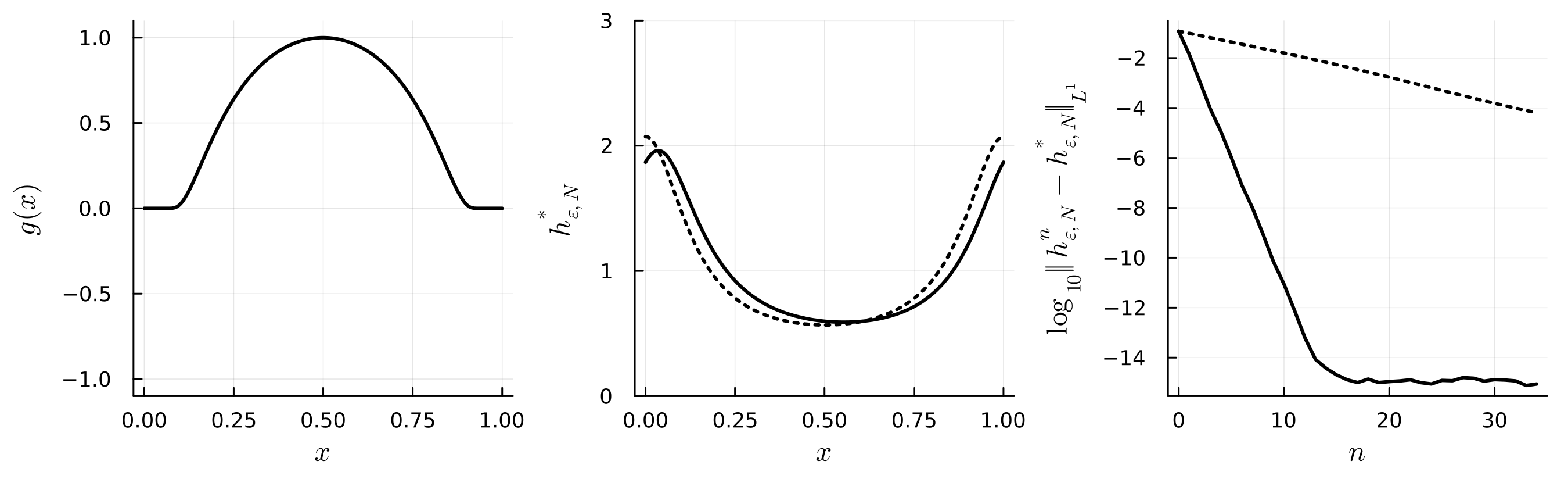}
    \caption{Left: graph of kernel $g$ representing coupling whose effect is a non-local translation to the right.  Centre: fixed point $h_{0,N}^*$ of the uncoupled system (dashed) and fixed point $h_{\eps,N}^*$ of the coupled system (solid). Right: $L^1$ error $\|h^n_{\eps,N}-h^*_{\eps,N}\|_{L^1}$ in $\log10$ scale vs iteration count $n$ for sequential iteration (dashed) and Newton iteration (solid), both initialised with $h_{0,N}^*$}
    \label{fig:translation}
\end{figure}
\begin{figure}
    \centering
    \includegraphics[width=\linewidth]{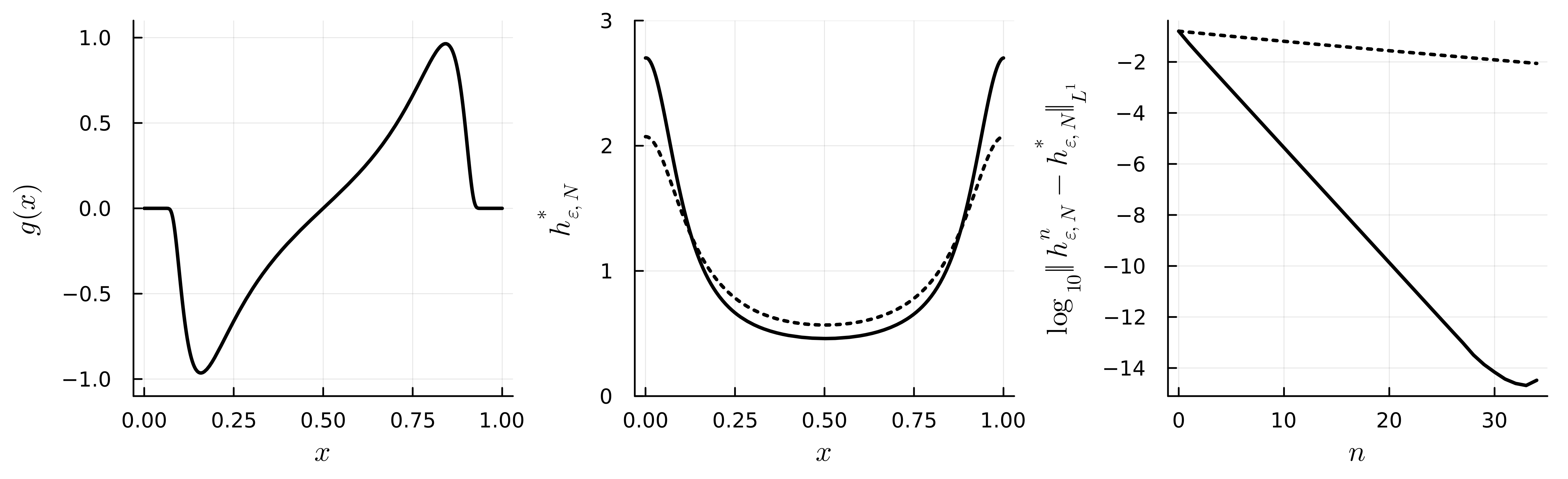}
    \caption{Left: graph of kernel $g$ representing coupling whose effect is non-local attraction.  Centre: fixed point $h_{0,N}^*$ of the uncoupled system (dashed) and fixed point $h_{\eps,N}^*$ of the coupled system (solid). Right: $L^1$ error $\|h^n_{\eps,N}-h^*_{\eps,N}\|_{L^1}$ in $\log10$ scale vs iteration count $n$ for sequential iteration (dashed) and Newton iteration (solid), both initialised with $h_{0,N}^*$.}
    \label{fig:attraction}
\end{figure}

Referring to \eqref{eq:couplingmap}, the nonnegativity of the kernel $g(x)=b(x)$ in Figure \ref{fig:translation}(left) and nonnegativity of the density $f$ means that this kernel represents a non-local translation to the right, with the peak translation effect occuring by coupling at a distance of 1/2 of the circular domain.
In contrast, the dipole structure of $g(x)=-b'(x)/5$ centred about $x=0$, as seen in Figure \ref{fig:attraction} (left), represents a (symmetric, due to the oddness of $g$) non-local attraction, with the peak coupling effects occurring at a distance of about 0.15.  
For each choice of kernel $g$ we construct the density-dependent coupled map $T_{h,\eps}$ according to \eqref{eq:couplingmap} and \eqref{eq:coupleddy} with  $\varepsilon=0.025$.

We now deploy our two numerical approaches to estimate $h_{\varepsilon}^*$, a fixed point of the nonlinear self-consistent transfer operator $\tilde{\mathcal{L}}_\varepsilon$.
Firstly, we use sequential iteration, initialising with $h_{\varepsilon,N}^0:=h_{0,N}^*$ and pushing forward according to the discrete iteration rule \eqref{eq:discreteiter} for 35 steps.
Secondly, again initialising with $h_{\varepsilon,N}^0:=h_{0,N}^*$, 
 we iterate with Newton scheme \eqref{eq:newtonupdate2} for 35 steps.
In Figure \ref{fig:translation}(centre) we see that the coupled fixed point has moved to the right by an amount roughly 0.05  
The coupled fixed point $h^*_{\eps,N}$ is also slightly smoothed relative to the uncoupled fixed point $h^*_{0,N}$ due to the coupling translating mass from the sticky fixed point at $x=0$.
In Figure \ref{fig:attraction}(centre) we see that the symmetric attractive coupling leads to greater symmetric density concentration near the already sticky fixed point of $T_0$ at $x=0$. 

To compute convergence rates to $h_{\varepsilon,N}^*$ we use the 35-step Newton estimate for $h_{\varepsilon,N}^*$, because the Newton updates are at machine precision.
The error decay $\|h_{\varepsilon,N}^{n}-h_{\varepsilon,N}^{*}\|_{L^1}$ is shown in Figures \ref{fig:translation}(right) and \ref{fig:attraction}(right) for both numerical schemes.
The experiment reported in Figure \ref{fig:attraction} took 20 seconds for sequential iteration and 60 seconds for Newton iteration\footnote{Computations performed on an Intel Core Ultra 7 165U laptop processor running Julia 1.12.4. In the first example using Fourier modes to order $N=64$ is sufficient, twenty iterations of Newton's method on a  takes only a few seconds.}.
Thus, for a three-fold increase in computation time, one achieves a relative error improvement of 14 orders of magnitude.

We can also investigate numerically how the error $\Vert h_{\varepsilon,N}^* - h_{\varepsilon}^*\Vert_{L^1}$ behaves with $N$.
In Figure \ref{fig:Nrate} we estimate $h_{\varepsilon}^*$ with $h_{\varepsilon,2^{10}}^*$ (which uses $>2000$ Fourier modes), computed with 50 Newton iterations, which is more than sufficient to reach machine precision.
We then compute the $L^1$ distances $\|h_{\varepsilon,N}^* - h_{\varepsilon,2^{10}}^*\Vert_{L^1}$ for $N=2^i$, $i=1,2,\ldots,7$, where each $h_{\varepsilon,N}^*$ is again computed with 50 Newton iterations.
In the log-log axes of Figure \ref{fig:Nrate} we see a rate comparable to $O(\log N/N)$.
\begin{figure}[H]
    \centering
    \includegraphics[width=0.5\linewidth]{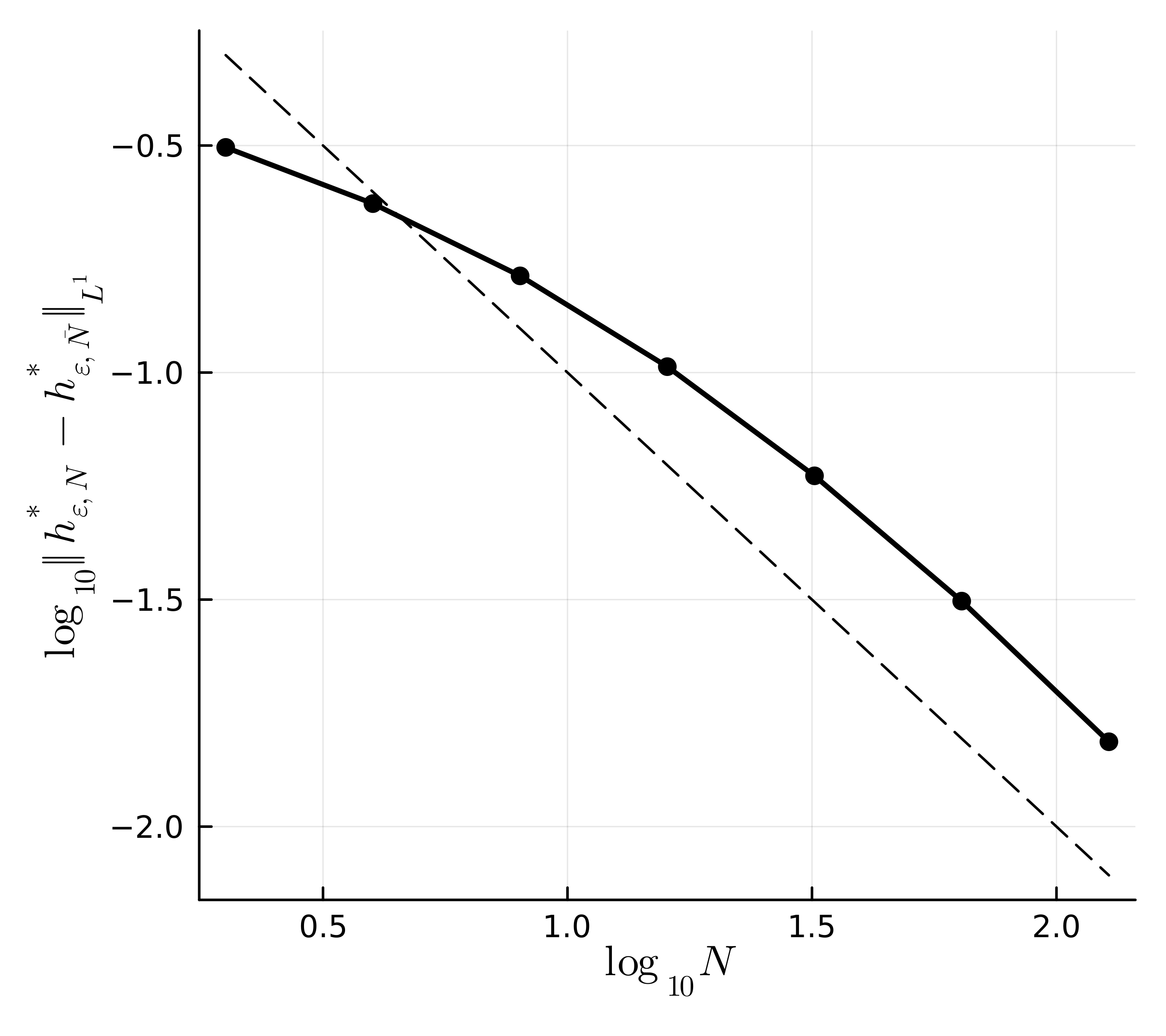}
    \caption{Behaviour of discrete fixed point errors $\|h_{\varepsilon,N}^* - h_{\varepsilon,2^{10}}^*\Vert_{L^1}$ for $N=2^i$, $i=1,2,\ldots,7$. This log-log plot is consistent with approximately $O(\log N/N)$ decay, with $1/N$ decay shown by the dashed line with slope of -1.}
    \label{fig:Nrate}
\end{figure}

\appendix
\section{Results from the literature on self-consistent transfer operators}

\begin{lemma}
    \label{introprop}
The following sequential Lasota-Yorke and continuity inequalities hold\footnote{In the following two inequalities we label different elements of $\mathcal{D}_1$ by a subscript; i.e., $f_j$ for some $j$. This should not be confused with the superscript notation used below in \eqref{eq:evolve}, $h^n$, for a state at time $n$.}: 
\begin{enumerate}
\item $\exists C_{LY}>0$ such that for all $T_{\eps,f_n},\dots,T_{\eps,f_1}\in\cT$, $h\in W^{i,1}$, $i=1,2,3,4$, and $n\in\mathbb N$
\begin{equation}\label{eq:seqLY}
\|\L_{ T_{\eps,f_n}\circ\dots\circ T_{\eps,f_1}} h\|_{W^{i,1}}\leq \lambda^{ni}\|h\|_{W^{i,1}}+ C_{LY}\|h\|_{W^{i-1,1}};
\end{equation}
\item $\exists C>0$ such that for all $T_{\eps,f_2},T_{\eps,f_1}\in\cT$, $f_1,f_2\in W^{i,1}$, $h\in W^{i+1,1}$, $i=0,1$,
\begin{equation}\label{eq:closeop}
\|(\L_{T_{\eps, f_1}}-\L_{T_{\eps,f_2}})h\|_{W^{i,1}} \leq  C\eps \|f_1-f_2\|_{W^{i,1}} \|h\|_{W^{i+1,1}}.
\end{equation}
\end{enumerate}
\end{lemma}
\begin{proof}
    The verification of \eqref{eq:seqLY} is standard, similar calculations can be found in \cite{galatolo22}. While for \eqref{eq:closeop}, one can easily show, see for instance \cite{bahetal23},
    $$\|(\L_{T_{\eps, f_1}}-\L_{T_{\eps,f_2}})h\|_{W^{i,1}} \leq C\|h\|_{W^{i+1,1}}d_{C^5}(T_{\eps,f_1},T_{\eps,f_2}).$$ Then using \eqref{eq:couplingmap} we get, for $j=0,\dots, 5$
$$|T_{\eps,f_1}^{(j)}(x)-T_{\eps,f_2}^{(j)}|\le \eps\|g^{(j)}\|_\infty\cdot\|f_1-f_2\|_1.$$
\end{proof}

\begin{theorem}\label{thm:gen}
 The following hold: 
\begin{enumerate}
    \item For each
    $\eps\in[0,\eps^*_0]$, $\tilde\L_\eps$ has a fixed point in $B_K$.    
\item There exists $\eps_1^* \in(0,\eps_0^*]$, $C>0$, $\theta\in(0,1)$ such that for all $\eps \in[0,\eps^*_1]$, $h\in W^{i,1}$, $i=1,2$, with $\int h=0$ and any sequence  $(f_j)_{j\in \mathbb{N}} \in (B_K)^{\mathbb{N}}$,
\begin{align*}
    \|\L_{T_{\eps,f_n}}\circ\cdots \circ \L_{T_{\eps, f_1}}h\|_{W^{i,1}}\le C \tilde \theta^n\|h\|_{W^{i,1}}.
\end{align*}

\item There exists $\eps_2^* \in(0,\eps_1^*]$,  such that for all $\eps \in[0,\eps^*_2] $ the operator $\tilde\L_\eps$ has a unique fixed point $h_\eps^*\in B_K$. Moreover, there exists $C_K>0$ and $\gamma\in (0,1)$ such that for all $\eps \in[0,\eps^*_2]$,  for all $h^0\in B_K$, we have
\[
\|{\tilde\L}_\eps^{n}(h^0)-h_\eps^*\|_{W^{1,1}}\le C_K \gamma^n;
\]
more explicitly,  
$$\|\mathcal{L}_{T_{\eps, h^{n-1}}}\circ\cdots\circ \mathcal{L}_{T_{\eps, h^{1}}}\circ\mathcal{L}_{T_{\eps, h^{0}}}h^{0} - h^*_\eps\|_{L^1}\le C_K \gamma^n.$$
\end{enumerate}
\end{theorem}
The proof of the above theorem is by now standard. See for instance \cite[Sections 3, 4 and 7]{galatolo22} for abstract results and applications to circle maps, \cite[Subsection 4.3]{keller00} for original work in the setting of circle maps, \cite[Theorem 1] {balintetal18} for a refinement of \cite{keller00} and \cite[Proposition 2.8, Theorem 2.9]{bahetal23} for results equivalent to the above theorem in a uniformly hyperbolic setting.  
\begin{remark}
Notice that any fixed point of $\tilde\L_\eps$, $h^*_\eps$, is an element of $W^{4,1}$. This follows from the definition of $\tilde\L_\eps$: 
 $$h^*_\eps=\tilde\L_\eps h_\eps=\L_{T_{\eps,h^*_\eps}}h^*_\eps.$$
 Thus, $h^*_\eps$ is the invariant density of some $T_{\eps,h^*_\eps}\in\mathcal T$. 
 One also has, due to $g\in C^5$, that $T_{\eps,f}\in C^5$ for any $f\in \mathcal{D}_1\subset L^1(\mathbb{S}^1)$. Therefore, $h^*_\eps\in W^{4,1}$.
\end{remark}

\section{An auxiliary lemma}

\begin{lemma}\label{lem:product}
    If $k\geq 1$ and $f,g\in W^{k,1}(\mathbb{S}^1,\mathbb{R})$, then 

\begin{align}\label{eq:sobolev}
    \|f\|_{\infty}\leq \|f\|_{W^{1,1}},
\end{align}
    and
\begin{align*}
    \|fg\|_{W^{k,1}}\leq 2^k\|f\|_{W^{k,1}}\|g\|_{W^{k,1}}.
\end{align*}
\end{lemma}

\begin{proof}
the proof of \eqref{eq:sobolev} is a consequence of Sobolev inequality for Sobolev maps on $\mathbb S^1$ : if $f,g\in W^{k,1}(\mathbb{S}^1,\mathbb{R})$, then $f,g\in C^0(\mathbb S^1)$ and $\|f\|_{\infty}\leq \|f\|_{W^{1,1}}$.
The proof of the other equation is done by induction. To initialize, let $f,g\in W^{1,1}$, notice that by \eqref{eq:sobolev}, 
\begin{align*}
    \|(fg)'\|_{L^1}&\leq \|f'g\|_{L^1}+\|fg'\|_{L^1}\\
&\leq \|f'\|_{L^1}\|g\|_{W^{1,1}}+\|f\|_{W^{1,1}}\|g'\|_{L^1}\\
\end{align*}
and since
\begin{align*}
    \|fg\|_{L^1}\leq \|f\|_{W^{1,1}}\|g\|_{L^1},
\end{align*}
we obtain by summing the two inequalities,
\begin{align*}
    \|fg\|_{W^{1,1}} &\leq 2\|f\|_{W^{1,1}}\|g\|_{W^{1,1}}.
\end{align*}
For the induction, assume the Lemma holds up to $k\in \mathbb N^*$ and any $f,g\in W^{k,1}$. Let $f,g\in W^{k+1,1}$,
\begin{align*}
    \|fg\|_{W^{k+1,1}}&\leq \|(fg)'\|_{W^{k,1}}+\|fg\|_{L ^1}\\
    &\leq \|f'g\|_{W^{k,1}}+\|fg'\|_{W^{k,1}}+\|fg\|_{L^1}\\
      &\leq 2^k\|f'\|_{W^{k,1}}\|g\|_{W^{k,1}}+2^k\|f\|_{W^{k,1}}\|g'\|_{W^{k,1}}+\|fg\|_{L^1}\\
      &\leq 2^k\|f'\|_{W^{k,1}}\|g\|_{W^{k,1}}+2^k\|f\|_{W^{k,1}}\|g'\|_{W^{k,1}}+\|f\|_{W^{k,1}}\|g\|_{L^1}\\
      &\leq 2^k\|f'\|_{W^{k,1}}\|g\|_{W^{k,1}}+2^k\|f\|_{W^{k,1}}\left(\|g'\|_{W^{k,1}}+\|g\|_{L^1}\right)\\
      &\leq 2^{k+1}\|f\|_{W^{k+1,1}}\|g\|_{W^{k+1,1}}.
\end{align*}

\end{proof}

\section{Acknowledgements}
This research was initiated during GF's visit to Loughborough University, which was supported by WB's EPSRC grant EP/V053493/1. An ARC Laureate Fellowship FL230100088 partially supported the  research of GF and fully supported the research of MP.
WB would like to acknowledge the Sydney Mathematical Research Institute (SMRI) for funding his visit to the University of Sydney and UNSW Sydney, and the Laureate Fellowship for partially funding his travel to MATRIX, visits during which further substantial progress was made. 
WB, GF and MP would like to thank the MATRIX Institute for Mathematical Research for hospitality.

\bibliographystyle{plain}
\bibliography{refs}
\end{document}